\documentclass[12pt]{article}  
\usepackage{amsmath}
\usepackage{amssymb}
\usepackage{theorem}
\usepackage{euscript}
\usepackage{pstricks}
\usepackage{authblk}
\usepackage{verbatim}
\usepackage{graphics}
\usepackage{graphicx}
\usepackage{color}

\usepackage{hyperref}
\hypersetup
{
	colorlinks,
	citecolor=black,
	filecolor=black,
	linkcolor=blue,
	urlcolor=blue
}
\hypersetup{linktocpage}
\newtheorem{theorem}{Theorem}[section]
\newtheorem{lemma}[theorem]{Lemma}

\numberwithin{equation}{section}

\theoremstyle{definition}
 
\theoremstyle{remark}

\newcommand{\brac}[1]{\left(#1\right)}

\newcommand{\brab}[1]{\left\{#1\right\}}

\newcommand{\bk}{{\boldsymbol{k}}}

\newcommand{\bm}{{\boldsymbol{m}}}

\newcommand{\bu}{{\boldsymbol{u}}}
\newcommand{\bv}{{\boldsymbol{v}}}
\newcommand{\bp}{{\boldsymbol{p}}}
\newcommand{\bq}{{\boldsymbol{q}}}
\newcommand{\br}{{\boldsymbol{r}}}
\newcommand{\bs}{{\boldsymbol{s}}}

\newcommand{\bx}{{\boldsymbol{x}}}

\newcommand{\by}{{\boldsymbol{y}}}

\newcommand{\bW}{{\boldsymbol{W}}}

\newcommand{\bone}{{\boldsymbol{1}}}

\newcommand{\bvarphi}{{\boldsymbol{\varphi}}}

\newcommand{\bxi}{{\boldsymbol{m}}}
\newcommand{\rd}{{\rm d}}

\def\ZZd{{\mathbb Z}^d}

\def\ZZ{{\mathbb Z}}

\def\RR{{\mathbb R}}
\def\RRd{{\mathbb R}^d}

\def\NN{{\mathbb N}}
\def\NNd{{\NN}^d}

\def\NN{{\mathbb N}}
\def\RR{{\mathbb R}}

\def\Bb{{\mathcal B}}

\def\NNd{{\mathbb N}^d}
\def\RRd{{\mathbb R}^d}

\def\TTd{{\mathbb T}^d}

\def\ZZd{{\mathbb Z}^d}

\def\Bb{{\mathcal B}}

\def\Ii{{\mathcal I}}

\def\Pp{{\mathcal P}}

\def\Ss{{\mathcal S}}

\def\ZZ{{\mathbb Z}}
\def\NN{{\mathbb N}}
\def\RR{{\mathbb R}}

\def\NNd{{\mathbb N}^d}
\def\RRd{{\mathbb R}^d}
\def\TTd{{\mathbb T}^d}

\def\supp{\operatorname{supp}}

\def\sign{\operatorname{sign}}

\def\Wap{W^r_p}

\newcommand{\norm}[2]{\left\|{#1}\right\|_{#2}}

\title{\sffamily Weighted sampling recovery  of functions \\  with   mixed smoothness} 

\author[a]{Dinh D\~ung}
\affil[a]{Information Technology Institute, Vietnam National University
	\protect\\
	144 Xuan Thuy, Cau Giay, Hanoi, Vietnam
	\protect\\
	Email: dinhzung@gmail.com}

\date{\today}
\begin{document}
\maketitle

\begin{abstract}
 We studied  linear weighted sampling  algorithms  and their optimality for approximate  recovery of   functions with mixed smoothness  on $\mathbb{R}^d$ from a  set of $n$ their sampled values. Functions to be recovered are in weighted Sobolev spaces $W^r_{p,w}(\mathbb{R}^d)$ of mixed smoothness, and the approximation error is measured by the norm of the weighted Lebesgue space $L_{q,w}(\mathbb{R}^d)$. Here, the weight $w$ is a tensor-product  Freud-type weight. The optimality of linear sampling algorithms is investigated in terms of  sampling $n$-widths. We constructed  linear sampling algorithms on  sparse grids of  sampled points  which form a step hyperbolic cross in the function domain, and which give upper bounds for the corresponding sampling $n$-widths. We proved that in the one-dimensional case, these algorithms realize the exact convergence rate of the $n$-sampling widths.

	
	
\end{abstract}

\section{Introduction}
\label{Introduction}

%
%

We begin with definitions of weighted  function spaces.  
Let 
\begin{equation} \label{w(bx)d}
	w(\bx):= w_{\lambda,\tau,\eta,a,b}(\bx) := \bigotimes_{i=1}^d w(x_i), \ \ \bx \in \RRd,
\end{equation}
be the tensor product of $d$ copies of the generating univariate Freud-type  weight
\begin{equation} \label{w(x)}
	w(x):= 	|x|^\tau (1 + |x|)^\eta \exp \brac{- a|x|^\lambda + b},
	 \ \ \lambda > 1,  \ a >0, \ \tau, \eta \ge 0, \ b \in \RR.
\end{equation} 
The most important parameters in the weight $w$ are $\lambda$ and $\tau$. The parameter $b$ which produces only a positive constant in the weight $w$  is introduced for a certain normalization for instance, for the standard Gaussian weight which is one of the most important weights. The weight $w$ has a singularity at $0$ if $\tau > 0$.
In what follows,  we fix the weight $w$ and hence the  parameters $\lambda,\tau, \eta, a, b$.

Let  $1\leq q<\infty$ and $\Omega$ be a Lebesgue measurable set on $\RRd$. 
We denote by  $L_{q,w}(\Omega)$ the weighted Lebesgue space  of all measurable functions $f$ on $\Omega$ such that the norm
\begin{align} \label{L-Omega}
\|f\|_{L_{q,w}(\Omega)} : = 
\bigg( \int_\Omega |f(\bx)w(\bx)|^q  \rd \bx\bigg)^{1/q}
\end{align}
is finite.

Put $\RRd_\tau = \RRd$ if $\tau = 0$, and $\RRd_\tau = (\RR \setminus \brab{0})^d$ if $\tau > 0$, where $\tau$ is the parameter in the definition \eqref{w(x)} of the generating univariate weight $w$. 
For $q=\infty$ and $\Omega = \RRd$, we define  the space $L_{\infty,w}(\RRd):=C_w(\RRd)$ of all measurable functions on $\RRd$ such that
\begin{equation*}
 f \in C(\RRd_\tau), \ \lim_{|x_j| \to \infty} f(\bx)w(\bx) = 0, \, j=1,...,d, 
\end{equation*}
for $\tau = 0$, and 
\begin{equation*}
 f \in C(\RRd_\tau), \
		\lim_{|x_j| \to \infty} f(\bx)w(\bx) = \lim_{x_j\to 0} f(\bx)w(\bx) = 0, \, j =1,...,d,
\end{equation*}
for $\tau > 0$
The norm in $L_{\infty,w}(\RRd)$ is defined by
\begin{equation*}
	\norm{f}{L_{\infty,w}(\RRd)}:=  \sup_{\bx\in \RRd_\tau} 
	| f(\bx)w(\bx)|. 
\end{equation*}

 For $r \in \NN$ and $1 \le p \le \infty$, the weighted  Sobolev space $W^r_{p,w}(\Omega)$ of mixed smoothness $r$  is defined as the normed space of all functions $f\in L_{p,w}(\Omega)$ such that the weak  partial derivative $D^{\bk} f$ belongs to $L_{p,w}(\Omega)$ for  every $\bk \in \NNd_0$ satisfying the inequality $|\bk|_\infty \le r$. The norm of a  function $f$ in this space is defined by
\begin{align} \label{W-Omega}
	\|f\|_{W^r_{p,w}(\Omega)}: = \Bigg(\sum_{|\bk|_\infty \le r} \|D^{\bk} f\|_{L_{p,w}(\Omega)}^p\Bigg)^{1/p}.
\end{align}

Let $\gamma$ be the standard $d$-dimensional Gaussian measure  with the density function 
$$
w_{\operatorname{g}}(\bx): = (2\pi)^{-d/2}\exp (- |\bx|^2/2).
$$
The well-known  spaces $L_p(\Omega;\gamma)$ and $\Wap(\Omega; \gamma)$ 
which are used in many applications, are defined  in the same way by replacing the norm \eqref{L-Omega} with the norm
$$
\|f\|_{L_p(\Omega; \gamma)} : = 
\bigg( \int_\Omega |f(\bx)|^p \gamma(\rd \bx)\bigg)^{1/p}
=
\bigg( \int_\Omega |f(\bx)\brac{w_{\operatorname{g}}}^{1/p}(\bx)|^p  \rd \bx\bigg)^{1/p}.
$$
Thus, the spaces $L_p(\Omega;\gamma)$ and $\Wap(\Omega; \gamma)$ coincide with  $L_{p,w}(\Omega)$ and $W^r_{p,w}(\Omega)$, where $w:= \brac{w_{\operatorname{g}}}^{1/p}$
for  a fixed $1 \le p < \infty$.  Throughout this paper, we use $\bW^r_{p,w}(\Omega)$ and $\bW^r_p(\Omega;\gamma)$ to denote the unit ball in the spaces $W^r_{p,w}(\Omega)$ and $W^r_p(\Omega;\gamma)$, respectively.

Let us formulate a setting of  linear sampling recovery problem. Let  $X$ be  a normed space of functions on $\Omega$. Given sample points $\bx_1,\ldots,\bx_k \in \Omega$, we consider the approximate recovery  of a continuous function $f$ on $\Omega$  from their values $f(\bx_1),\ldots, f(\bx_k)$ by a linear sampling algorithm $S_k$
on $\Omega$ of the form
\begin{equation} \label{S_k}
	S_k(f): = \sum_{i=1}^k  f(\bx_i) h_i, 
\end{equation}
where  $h_1,\ldots,h_k$ are given continuous functions on $\Omega$.  For convenience, we assume that some of the sample points $\bx_i$ may coincide. The approximation error is measured by the norm 
$\|f - S_k (f)\|_X$.  Denote by $\Ss_n$  the family of all linear sampling algorithms $S_k$ of the form \eqref{S_k} with $k \le n$.
Let $F \subset X$ be a set of continuous functions on $\Omega$.   To study the optimality  of linear  sampling algorithms  from $\Ss_n$ for  $F$ and their convergence rates we use  the  (linear) sampling $n$-width
\begin{equation} \label{rho_n}
\varrho_n(F, X) :=\inf_{S_n \in \Ss_n} \ \sup_{f\in F} 
\|f - S_n (f)\|_X.
\end{equation}

Notice that any function $f \in W^r_{p,w}(\RRd)$ is equivalent  in the sense of the Lesbegue measure to a continuous (not necessarily bounded) function on $\RRd_\tau$ (see \cite[Lemma 3.1]{DD2023} for $\tau = 0$; in the case $\tau >0$, this equivalence can be proven similarly). Hence throughout the present paper,  we always assume that $W^r_{p,w}(\RRd) \subset C_w(\RRd)$, and that  $W^r_{p,w}(\RRd)$ coincides with $W^r_{p,w}(\RRd_\tau)$ in the sense of the Lebesgue measure.  Therefore, linear  sampling algorithms of the form \eqref{S_k} with $\Omega = \RRd_\tau$ are well-defined for functions $f \in W^r_{p,w}(\RRd)$. 

There is a large number of works devoted to the problem of (unweighted) linear sampling recovery of functions having a mixed smoothness on a compact domain. 
Smolyak sparse-grid sampling algorithms for compact domains  (cube or torus) have been widely and efficiently used in both approximation theory and numerical analysis, especially, in sampling recovery for functions having a mixed smoothness.   The reader can consult \cite{BG2004,DD2023-survey,DTU18B,NoWo10,Tem18B} for a survey and bibliography. 
The optimal convergence rate of sampling $n$-widths  is one of central problems in sampling recovery of functions having a mixed smoothness.  We refer the reader to \cite{DD2023-survey,DTU18B} for  a survey and bibliography on the optimal  convergence rate of  the sampling $n$-widths $\varrho_n\big(\bW^r_p(\TTd), L_q(\TTd)\big)$, where $\bW^r_p(\TTd)$ is the unit ball in the Sobolev space $W^r_p(\TTd)$ of functions on the $d$-dimensional torus $\TTd$ having mixed smoothness $r$. 
It is noteworthy that the recent results of \cite{DKU22} on  inequality between the sampling and Kolmogorov $n$-widths of the unit ball in a reproducing kernel Hilbert subspace immediately yield in a nonconstructive way the exact convergence rate of
 $\varrho_n(\bW^r_2(\TTd),L_2(\TTd))$. This solved the outstanding Open Problem 4.1 in \cite{DTU18B} for the particular case $p=2$.

 The problem of optimal sampling recovery of functions on $\RRd$ equipped with Gaussian measure has been investigated in \cite{DK2022}. We proved in a non-constructive way the exact convergence rate  of the  sampling $n$-widths 
 \begin{equation} 	\label{AsympSampling-p=2}
 	\varrho_n\big(\bW^r_p(\RRd; \gamma), L_2(\RRd; \gamma)\big) 
 	\asymp
 	n^{-r} (\log n)^{r(d-1)}
 \end{equation}
 for $2 < p \le \infty$,  and
 \begin{equation} 	\label{AsympSampling-p=2}
 	\varrho_n\big(\bW^r_2(\RRd; \gamma), L_2(\RRd; \gamma)\big) 
 	\asymp
 	n^{-r/2} (\log n)^{r(d-1)/2}
 \end{equation}
 which is obtained  by using  the inequalities between  sampling widths and Kolmogorov widths~\cite{DKU22} mentioned above, and the exact convergence rate of $d_n\big(\bW^r_p(\RRd; \gamma), L_q(\RRd; \gamma)\big)$ proven in the same paper \cite{DK2022}.  
A related problem on hyperbolic cross weighted polynomial approximation based on  based on de la Vall\'ee Poussin sums of the relevant orthonormal polynomial expansion, has been treated in \cite{DD2024}.
 
In the present paper, we are interested in constructing  sparse-grid linear sampling  algorithms  for approximate  recovery of   functions with mixed smoothness  on $\mathbb{R}^d$ in the following setting: 
  functions to be recovered are in weighted Sobolev spaces $W^r_{p,w}(\mathbb{R}^d)$ of mixed smoothness $r$, and the approximation error is measured by the norm of the weighted Lebesgue space $L_{q,w}(\mathbb{R}^d)$.
The parameters $p$ and $q$ may be different and range in $[1, \infty]$. The asymptotic  optimality of linear sampling algorithms is investigated in terms of  the relevant sampling widths. We construct sparse-grid sampling algorithms which give upper bounds of the  sampling $n$-widths 
$\varrho_n(\boldsymbol{W}^r_{p,w}(\mathbb{R}^d), L_{q,w}(\RRd))$. In, the case $d=1$, these algorithms will realize the exact convergence rate of the sampling widths. 
In particular, we develop a counterpart  of  Smolyak sparse-grid sampling algorithms  which realize the upper bounds.
 
We briefly describe the results of the present paper.  Before doing this we emphasize that  our setting of problem of optimal weighted sampling recovery  naturally comes from the classical theory of weighted approximation (for knowledge and bibliography see, e.g., \cite{Mha1996B}, \cite{Lu07B},  \cite{JMN2021}).

Throughout the present paper, for given $p, q \in [1, \infty]$ and the parameter $\lambda > 1$ in the definition \eqref{w(x)} of the generating univariate weight $w$,  we make use of  the notations
$$
r_\lambda:= (1 - 1/\lambda)r, \quad r_{\lambda,p,q}:= r_\lambda - 	\delta_{\lambda,p,q},
$$ 
and (with the convention $1/\infty := 0$) 
\begin{equation}\label{delta_{lambda,p,q}}
	\delta_{\lambda,p,q} 
	:=
	\begin{cases}
		(1-1/\lambda)(1/p - 1/q)&  \ \ \text{if} \ \ p \le q, \\
		(1/\lambda)(1/q - 1/p)&  \ \ \text{if} \ \ p > q.
	\end{cases}	 
\end{equation}

Let  $1<p <\infty$, $1\le q \le \infty$ and $r_{\lambda,p,q} >0$. 	Then  with some restrictions on the weight $w$ and $p,q$, we prove that 
\begin{equation}\label{UpperB-introduction}
	\varrho_n(\bW^r_{p,w}(\RRd), L_{q,w}(\RRd)) 
	\ll 
	\begin{cases}
		n^{-r_{\lambda}} (\log n)^{(r_{\lambda} + 1)(d-1)} &  \ \ \text{if} \ \ p = q, \\
		n^{-r_{\lambda,p,q}} (\log n)^{(r_{\lambda,p,q} + 1/q)(d-1)} &  \ \ \text{if} \ \ p \not= q<\infty,
	\\
	n^{-r_{\lambda,p,q}} (\log n)^{(r_{\lambda,p,q} + 1)(d-1)} &  \ \ \text{if} \ \ p \not= q=\infty;
	\end{cases}	 
\end{equation}
and
\begin{equation}\label{LowerB-Introduction}
	\varrho_n(\bW^r_{p,w}(\RRd), L_{q,w}(\RRd)) 
	\gg
	\begin{cases}
		n^{-r_{\lambda}} (\log n)^{r_{\lambda}(d-1)} &  \ \ \text{if} \ \ p = q, \\
		n^{-r_{\lambda,p,q}} (\log n)^{r_{\lambda,p,q}(d-1)} &  \ \ \text{if} \ \ p < q, \\
		n^{-r_{\lambda,p,q}} &  \ \ \text{if} \ \ p > q.
	\end{cases}	 
\end{equation}
In  the one-dimensional  case, we prove the exact convergence rate
\begin{equation}\label{rho_n(W)-introduction-d=1}
	\varrho_n(\bW^r_{p,w}(\RR),  L_{q,w}(\RR)) 
	\asymp 
	n^{-r_{\lambda,p,q}}.
\end{equation}
We also obtain upper and lower bound of  $\varrho_n(\bW^r_{\infty,w}(\RRd), L_{q,w}(\RRd)) $	in the case when $1\le q \le \infty$ and $r_{\lambda,\infty,q} >0$: 
	\begin{equation}\label{rho_n-d=1,p=infty-Introduction}
		n^{- r_{\lambda,\infty,q}}
		\ll
		\varrho_n\big(\bW^r_{\infty,w}(\RR),  L_{q,w}(\RR)\big) 
		\ll
		n^{- r_{\lambda,\infty,q}} \log n.
	\end{equation}
	
 Some results on univariate weighted interpolation  from \cite{MN2010} are employed in proving the upper bounds \eqref{UpperB-introduction}. The Smolyak sampling algorithms performing these upper bounds  are constructed  on  sparse grids of  sampled points  which form \emph{a step hyperbolic cross in the function domain} $\RRd$. This development of the construction of sparse grids and the associated algorithms  was proposed first  by the author \cite{DD2023} 
 for numerical weighted integration of functions having weighted mixed smoothness.  It is remarkable to notice that these sparse grids are completely different from  classical hyperbolic crosses in  frequency domains and Smolyak sparse grids   in compact function domains (see Figures \ref{Fig-1} and \ref{Fig-2}). To prove the lower bounds in  \eqref{LowerB-Introduction} and \eqref{rho_n(W)-introduction-d=1} in the case $p < q$ we adopt a traditional technique to construct for arbitrary $n$ sampled points a fooling function vanishing at these points. The lower bounds  in  \eqref{LowerB-Introduction}  and \eqref{rho_n(W)-introduction-d=1}  the case $p \ge q$ are derived from the lower bound of  Kolmogorov widths, a Bernstein-type inequality by using  a discretization technique. 

  The gap between the upper bounds (convergence rates) \eqref{UpperB-introduction} and the lower bounds in \eqref{LowerB-Introduction} for $d\ge 2$ (and $d=1$, $p=\infty$) is a logarithmic factor 
 which may be different depending on various values of $p$ and $q$.	 For all $p,q$ the main parameter $r_{\lambda,p,q}$ in the convergence rates is strictly smaller than  
 $r - (1/p -1/q)_+$ which is the main parameter  in the convergence rate of the respective  unweighted sampling $n$-widths. 
 
The paper is organized as follows. In Section \ref{Univariate sampling recovery}, we prove the exact convergence rate 
 of $\varrho_n(\bW^r_{p,w}(\RR),  L_{q,w}(\RR))$ and construct asymptotically optimal linear sampling algorithms. In Section \ref{Multivariate sampling recovery},  we prove upper and lower bounds of $\varrho_n(\bW^r_{p,w}(\RRd),  L_{q,w}(\RRd)) $ for $d \ge 2$, and construct linear sampling algorithms which give the upper bound. 

\medskip
\noindent
{\bf Notation.} 
 Denote   $\bx=:\brac{x_1,...,x_d}$ for $\bx \in \RRd$; $\bone:= (1,...,1) \in \RRd$; 
$|\bx|_p:= \brac{\sum_{j=1}^d |x_j|^p}^{1/p}$ $(1 \le p < \infty)$ and $|\bx|_\infty:= \max_{1\le j \le d} |x_j|$ with the abbreviation: $|\bx|:= |\bx|_2$.  For $\bx, \by \in \RRd$, the inequality $\bx \le \by$ ($\bx < \by$) means $x_i \le y_i$ ($x_i < y_i$) for every $i=1,...,d$.  For $x \in \RR$, denote $\sign (x):= 1$ if $x \ge 0$, and $\sign (x):= -1$ if  $x < 0$. We use letters $C$  and $K$ to denote general 
positive constants which may take different values. For the quantities $A_n(f,\bk)$ and $B_n(f,\bk)$ depending on 
$n \in \NN$, $f \in W$, $\bk \in \ZZd$,  
we write  $A_n(f,\bk) \ll B_n(f,\bk)$, $f \in W$, $\bk \in \ZZd$ ($n \in \NN$ is specially dropped),  
if there exists some constant $C >0$ independent of $n,f,\bk$  such that 
$A_n(f,\bk) \le CB_n(f,\bk)$ for all $n \in \NN$,  $f \in W$, $\bk \in \ZZd$ (the notation $A_n(f,\bk) \gg B_n(f,\bk)$ has the obvious opposite meaning), and  
$A_n(f,\bk) \asymp B_n(f,\bk)$ if $S_n(f,\bk) \ll B_n(f,\bk)$
and $B_n(f,\bk) \ll S_n(f,\bk)$.  Denote by $|G|$ the cardinality of the set $G$. 

	\section{Univariate sampling recovery }
\label{Univariate sampling recovery}

In this section,   for the one-dimensional case, we prove the exact convergence rate of  sampling $n$-widths $\varrho_n\big(\bW^r_{p,w}(\RR), L_{q,w}(\RR)\big)$ and construct asymptotically optimal linear sampling linear algorithms for $1 < p < \infty$ and $1 \le q \le \infty$. We also give some upper and lower bounds of $\varrho_n\big(\bW^r_{\infty,w}(\RR), L_{q,w}(\RR)\big)$ for $1 \le q \le \infty$.
 
Let the univariate generalized Freud weight $v$ be defined by
\begin{equation} \label{v(x)}
	v(x):= 	|x|^\mu \exp \brac{- 2a|x|^\lambda + 2b}, \ \ \mu \ge -1,
\end{equation}
where $\lambda , a, b$ are the same as in the definition \eqref{w(x)} of the weight $w$. 
 Let $(p_m)_{m \in \NN_0}$ be the sequence of orthonormal polynomials with respect to the  weight $v$. 
For a given even number $m \in \NN$, denote by $x_{m,k}$, $1 \le k \le  m/2, $ the positive zeros of 	$p_m$, and by $x_{m,-k} = - x_{m,k}$ the negative ones (since $m$ is even, $0$ is not a zero of $p_m$). These zeros are located as 
\begin{equation}\label{zeros-location}
-a_m + \frac{Ca_m}{m^{2/3}} < x_{m, -m/2 } < \cdots 
< x_{m,-1} < x_{m,1} < \cdots <  x_{m, m/2 } < a_m - \frac{Ca_m}{m^{2/3}}, 
\end{equation}
with a positive constant $C$ independent of $m$ (see, e. g., \cite[(4.1.32)]{JMN2021}). Here $a_m$ is the Mhaskar-Rakhmanov-Saff number which is
\begin{equation}\label{a_m}
	a_m :=(C_\lambda m)^{1/\lambda} \asymp m^{1/\lambda}, \ \ C_\lambda:= \frac{2^{\lambda - 1} \Gamma(\lambda/2)^2}{\Gamma(\lambda)},
\end{equation}
and $\Gamma$ is the gamma function (see, e.g., \cite[(4.1.4)]{JMN2021}).

 Throughout this paper, we fix a number $\rho$ with $0 < \rho< 1$, and denote by $j(m)$ the smallest integer satisfying $x_{m,j(m)} \ge \rho a_m$. 
It is useful to remark that
\begin{equation} \label{Delta_k}
 x_{m,j(m)} \, \asymp \, m^{1/\lambda}, \  m \in \NN_0; \quad
\Delta_{m,k} \, \asymp  \, \frac{a_m}{m} \asymp m^{1/\lambda -1}, \ \ |k| \le j(m), 
\end{equation} 
where $\Delta_{m,k}:= x_{m,k} - x_{m,k-1}$ is the distance between consecutive zeros of  the polynomial $p_m$. 
The first relation follows from the definition, \eqref{a_m} and \eqref{zeros-location}. For the second relations see \cite[(4.1.47)]{JMN2021}. 
From their proofs there, one can easily see that they are still hold true for the general case of the weight $w$. By  \eqref{zeros-location} and \eqref{Delta_k}, for $m$ sufficiently large we have that 
\begin{equation} \label{<j(m)<}
Cm \le j(m) < m/2
\end{equation} 
with a positive constant $C$ depending on $\lambda, a, b$ and $\rho$ only. 

For an even $m \in \NN$, let $\Pp^*_{m+1}$ be the subspace of $\Pp_{m+1}$ defined by
\begin{equation} \label{}
\Pp^*_{m+1}: = \brab{\varphi \in \Pp_{m+1}: \varphi(\pm a_m) = \varphi(x_{m,k}) = 0, \ |k| > j(m)}. 
\end{equation} 
Here $\Pp_m$ denotes the space of polynomials of degree at most $m$.
For $|k| \le j(m)$, we define
\begin{equation} \label{}
	\ell_{m,k}(x): = \frac{p_m(x)}{p'(x_{m,k})(x - x_{m,k})} \frac{a_m - x^2}{a_m - x_{m,k}^2}.
\end{equation}
The polynomials $\ell_{m,k}$ belong to $\Pp_{m+1}$.  The  dimension of the space $\Pp^*_{m+1}$ is $2j(m)$ and the polynomials $\brab{\ell_{m,k}}_{|k|\le j(m)}$ constitute a basis of the space $\Pp^*_{m+1}$, i.e., every polynomial $\varphi \in \Pp^*_{m+1}$ can be represented as 
\begin{equation} \label{varphi=}
\varphi = \sum_{|k| \le j(m)}  \varphi(x_{m,k}) \ell_{m,k}. 
\end{equation}
We extend the formula \eqref{varphi=} to the functions $f$ in $C_w(\RR)$ as the interpolation operator 
$$I_m: C_w(\RR) \to \Pp^*_{m+1}$$ 
by defining
\begin{equation} \label{I_m f}
	I_m (f) := \sum_{|k| \le j(m)}  f(x_{m,k}) \ell_{m,k}. 
\end{equation}
 
The  operator $I_m$ is a projector from $ C_w(\RR)$ onto $\Pp^*_{m+1}$, and $I_m f$ belongs to $\Pp_{m+1}$. The polynomial $I_m f$ interpolates $f$ at the zeros $x_{m,k}$ for $|k| \le j(m)$, and vanishes at the points $\pm a_m$ and the zeros $x_{m,k}$ for $|k| >  j(m)$.
Also,  the number $2j(m)$ of interpolation points  in $I_m$ is strictly smaller than $m$.   However, due to \eqref{<j(m)<} it has the convergence rate as $2 j(m)\asymp m$ when $m$ going to infinity. The space of polynomials $\Pp^*_{m+1}$ and the interpolation operator $I_m$ have been introduced in \cite{MN2010}.
\\
\\
\noindent
{\bf Condition A}. The number $p$ with $1 < p < \infty$ and the numbers $\tau\ge 0$, $\eta \ge 0$, $\mu \ge -1$ associated with the weights $w$ and $v$, satisfy 
\begin{itemize}
	\item[(\rm i)] $\tau + 1/p$ is not an integer;
\item[(\rm ii)] $- 1/p < \tau - \mu/2 < 1 - 1/p - \eta$.
\end{itemize}

Throughout this and the next sections, for convenience of the presentation, we assume without mention that  Condition~A holds for a given 
$p$ with $1 < p < \infty$ in the assumptions of lemmata and theorems.

The following two lemmata were proven in \cite[Theorem 3.7, Lemma 3.3]{MN2010}.
\begin{lemma} \label{lemma:I_m}
Let $1 < p <\infty$.	Then  for every $m \in \NN$,
	\begin{equation}\label{|f - I_m f|<}
		\|f - I_m (f)\|_{L_{p,w}(\RR)} 
		\le
	C m^{- r_\lambda} \|f \|_{W^r_{p,w}(\RR)} \ \ \forall f \in W^r_{p,w}(\RR).  
	\end{equation}
\end{lemma}
\begin{lemma} \label{lemma:MarcinkiewiczInequality}
	Let $1 < p <\infty$.
	Then there holds the equivalence  
	\begin{equation}\label{MarcinkiewiczInequality}
		\|\varphi\|_{L_{p,w}(\RR)} 
		\asymp
		 \brac{m^{1/\lambda -1} \sum_{|k|\le j(m)} |\varphi(x_{m,k}) w(x_{m,k})|^p}^{1/p} \ \ 
		 \forall  \varphi \in \Pp^*_{m+1} \ \ \forall m \in \NN.  
	\end{equation}
\end{lemma}
\begin{lemma} \label{lemma:B-NInequality}
	Let $1 \le  p,q \le \infty$. Then we have the following.
\begin{itemize}
	\item[\rm{(i)}] 
	There holds the Bernstein-type inequality  
	\begin{equation}\label{BernsteinInequality}
		\|\varphi'\|_{L_{p,w}(\RR)} 
		\ll
		m^{1- 1/\lambda}\|\varphi\|_{L_{p,w}(\RR)}	 	 
		\ \ \ \
		\forall  \varphi \in \Pp_m, \ \ \forall m \in \NN.  
	\end{equation}
	\item[\rm{(ii)}] For any fixed $\delta  > 0$, there exists a constant $C_\delta > 0$ such that
	\begin{equation}\label{NikolskiiInequality}
		\|\varphi\|_{L_{p,w}(\RR)} 
		\le
	C_\delta \|\varphi\|_{L_{p,w}(I^\delta_m)}	 	 
		\ \ \ \
		\forall  \varphi \in \Pp_m, \ \ \forall m \in \NN, 
	\end{equation}		
where $I^\delta_m:= [-a_m, - \delta a_m/m] \cup [\delta a_m/m, a_m]$	
		\item[\rm{(iii)}] For $1 \le p < q \le \infty$, there holds the Nikol'skii-type inequality
			\begin{equation}\label{NikolskiiInequality_p<q}
			\|\varphi\|_{L_{q,w}(\RR)} 
			\ll
			m^{(1-1/\lambda)(1/p - 1/q)}\|\varphi\|_{L_{p,w}(\RR)}	 	 
			\ \ \ \
			\forall  \varphi \in \Pp_m, \ \ \forall m \in \NN.  
		\end{equation}		
	\item[\rm{(iv)}] For $1 \le q < p \le \infty$, there holds the Nikol'skii-type inequality
\begin{equation}\label{NikolskiiInequality_p>q}
	\|\varphi\|_{L_{q,w}(\RR)} 
	\ll
	m^{(1/\lambda)(1/q - 1/p)}\|\varphi\|_{L_{p,w}(\RR)},	 	 
	 \ \
	\forall  \varphi \in \Pp_m, \  \forall m \in \NN.  
\end{equation}		
\end{itemize}	
\end{lemma}

\begin{proof}
The claims (i)--(iii) were proven in \cite[pp. 106, 107]{MS2007}
(see also \cite[Lemma 4.1.3]{JMN2021}). The claim (iv) can be obtained from the claim (iii) and \eqref{a_m} by applying the H\"older inequality for $\nu = p/q >1$ and $1/\nu + 1/\nu' =1$. Indeed, for a fixed $\delta > 0$, we have
\begin{equation}
	\begin{split}\label{}
		\int_{\RR} |\varphi(x) w(x)|^q \rd x
		&\ll 
		(2a_m(1 - \delta/m ))^{1/\nu'}\brac{\int_{I^\delta_m} |\varphi(x) w(x)|^{q \nu} \rd x}^{1/\nu}\\
		& \ll
		m^{q(1/\lambda)(1/q - 1/p)}\brac{\int_{\RR} |\varphi(x) w(x)|^p \rd x}^{q/p}.
	\end{split}
\end{equation}
\hfill	
\end{proof}	
	
	Notice that Condition A combines necessary and sufficient conditions  for  Lemmas~\ref{lemma:I_m}--\ref{lemma:B-NInequality}, which are used in the proofs of all the results. For example, Condition A(ii) is a necessary and sufficient condition for one of the inequalities in Lemma \ref{lemma:MarcinkiewiczInequality} (see \cite[Lemma 3.3]{MN2010}), and Condition A is a sufficient condition for Lemma~\ref{lemma:I_m} (see \cite[Theorem 3.7]{MN2010}). For  detail, see \cite{MN2010}. To streamline the presentation,  all these conditions are encapsulated into Condition A, which serves as an assumption for  all the statements in what follows. 

For every $k \in \NN_0$, let $m_k$ be the largest number such that $m_k + 1 \le 2^k$.  Then by \eqref{<j(m)<} we have 
$2^k \asymp 2j(m_k) \le 2^k$. 
For the sequence of sampling operators $\brac{I_{m_k}}_{k \in \NN_0}$ with 
$
I_{m_k} \in \Ss_{2^k},
$
from Lemma \ref{lemma:I_m} with $s=0$ it follows that
\begin{equation}\label{IntErrorR2}
	\big\|f - I_{m_k}(f)\big\|_{L_{p,w}(\RRd)}  
	\leq C 2^{- r_\lambda k} \|f\|_{W^r_{p,w}(\RR)}, 
	\ \  \ k \in \NN_0,  \ \ f \in W^r_{p,w}(\RR) \ \ (1 < p < \infty).
\end{equation}
We define  the one-dimensional operators for $k \in \NN_0$
\begin{equation}\label{Delta}
	\Delta_k:=  I_{m_k} - I_{m_{k-1}}, \  k >0, \ \ \Delta_0:= I_1. 
\end{equation}
\begin{lemma} \label{lemma:Delta_k}
	Let $1< p < \infty$, $1\le q \le \infty$ and $r_{\lambda,p,q} >0$.   Then we have that
	for every $ f \in W^r_{p,w}(\RR)$,
	\begin{equation}\label{Series2}
		f
		= 	\sum_{k \in \NN_0}\Delta_k(f) 
	\end{equation}
	with absolute convergence in the space $L_{q,w}(\RR)$  of the series. 
	Moreover,
	\begin{equation}\label{Delta_kf}
		\big\|\Delta_k (f)\big\|_{L_{q,w}(\RR)}  
		\leq C 2^{-  r_{\lambda,p,q} k} \|f\|_{W^r_{p,w}(\RR)}, 
		\ \  k \in \NN_0.
	\end{equation}
\end{lemma}

\begin{proof}
Let 	$ f \in W^r_{p,w}(\RR)$.
Since $\Delta_k f \in \Pp_{m_k +1}$ by the claims (iii) and (iv) of  Lemma \ref{lemma:B-NInequality} we have that
		\begin{equation}\label{Delta_kf<<}
		\big\|\Delta_k(f)\big\|_{L_{q,w}(\RR)}  
		\ll 2^{\delta_{\lambda,p,q} k} 	\big\|\Delta_k (f)\big\|_{L_{p,w}(\RR)},
		\ \   k \in \NN_0.
	\end{equation}	
	By  \eqref{IntErrorR2} and \eqref{Delta} we have that for every $f \in W^r_{p,w}(\RR)$ and $k \in \NN_0$,
	\begin{equation}\nonumber
		\begin{aligned}
			\big\|\Delta_k (f)\big\|_{L_{p,w}(\RR)} 
			& \le  \big\|f  - I_{2 j(m_k)}(f)\big\|_{L_{p,w}(\RR)}  
			+ \big\|f  - I_{2 j(m_{k-1})}(f)\big\|_{L_{p,w}(\RR)} 
			\\
			&\ll 
			  2^{- r_\lambda k} \|f\|_{W^r_{p,w}(\RR)},
		\end{aligned}
	\end{equation}
	which together with \eqref{Delta_kf<<} proves \eqref{Delta_kf} and hence  the absolute convergence of the series in \eqref{Series2} follows. 
	The equality in \eqref{Series2} is implied from \eqref{|f - I_m f|<}  and the equality 
	\begin{equation}\label{I_{m_k}}
		I_{m_k}
		= 	\sum_{s \le k} \Delta_s.
	\end{equation}	
	\hfill	
\end{proof}


We  recall some well-known inequalities  which are quite useful  for lower estimation of  sampling $n$-widths. 
From the definition \eqref{rho_n} we have the following result.
If $F$ is a set of continuous functions on $\RRd$ and $X$ is a normed space of functions on $\RRd$,	then we have
\begin{equation} \label{rho_n>}
	\varrho_n(F, X) \ge \inf_{\brab{\bx_1,...,\bx_n} \subset \RRd} \ \sup_{f\in F: \ f(\bx_i)= 0,\ i =1,...,n} 
	\|f \|_X.
\end{equation}

Let $n \in \NN$ and 
let $X$ be a Banach space and $F$ a central symmetric compact set in $X$.
Then the Kolmogorov $n$-width  of $F$ is defined by
\begin{equation*}
	d_n(F,X):= \inf_{L_{n}}\sup_{f\in F}\inf_{g\in L_n}\|f-g\|_X, 
\end{equation*}
where the left-most  infimum is  taken over all  subspaces $L_{n}$ of dimension  $\le n$ in $X$.  
If $X$ is a normed space of functions on $\Omega$ and $F \subset X$ is a set of continuous functions on $\Omega$,  then from the definitions we have
\begin{equation} \label{rho_n > d_n}
	\varrho_n(F, X) \ge  d_n(F, X).
\end{equation}

\begin{theorem} \label{thm:I_m}
	Let $1< p < \infty$, $1\le q \le \infty$ and $r_{\lambda,p,q} >0$.   For any $n \in \NN$, let $k(n)$ be the largest integer such that $2^{k(n)}  \le n$. Then the sampling operators	$I_{m_{k(n)}} \in \Ss_n$, $n \in \NN$, are  asymptotically optimal for the sampling $n$-widths $\varrho_n\big(\bW^r_{p,w}(\RR),  L_{q,w}(\RR)\big)$ and
		\begin{equation}\label{rho_n-d=1}
			\varrho_n\big(\bW^r_{p,w}(\RR),  L_{q,w}(\RR)\big) 
			\asymp 
				\sup_{f\in \bW^r_{p,w}(\RR)} \big\|f - I_{m_{k(n)}}(f)\big\|_{ L_{q,w}(\RR)} 
			\asymp
			n^{- r_{\lambda,p,q}}.
		\end{equation}
\end{theorem}
\begin{proof}
 By using Lemma \ref{lemma:Delta_k} we derive  for $f\in \bW^r_{p,w}(\RR)$ and $n \in \NN_0$,
\begin{equation}\label{}
	\begin{split}\label{}
\big\|f - I_{m_{k(n)}}(f)\big\|_{ L_{q,w}(\RR)} 
		& = 
		\Bigg\|\sum_{k > k(n)} \Delta_k (f)\Bigg\|_{ L_{q,w}(\RR)}
		\le 
		\sum_{k > k(n)} \big\|\Delta_k (f)\big\|_{ L_{q,w}(\RR)}
		 \\
		&\ll
		\sum_{k > k(n)}  2^{-  r_{\lambda,p,q} k} \|f\|_{W^r_{p,w}(\RRd)} \\
		&\le
		2^{-  r_{\lambda,p,q} k(n)} \sum_{k > k(n)}  2^{-  r_{\lambda,p,q} (k - k(n))} 
		\ll
		n^{- r_{\lambda,p,q}}, 
	\end{split}
\end{equation}
	which proves the upper bound in \eqref{rho_n-d=1}.
		 
In order to prove the lower bound in the case $p < q$ in \eqref{rho_n-d=1} we employ the inequality \eqref{rho_n>}. Let $\brab{m_1,...,m_n} \subset \RR$ be arbitrary $n$ points. For a given $n \in \NN$, we put $\delta = n^{1/\lambda - 1}$ and $t_j = \delta j$, $j \in \NN_0$. Then there is $i \in \NN$ with 
$n + 1 \le  i \le 2n + 2$ such that the interval $(t_{i-1}, t_i)$ does not contain any point from the set $\brab{m_1,...,m_n}$.
Take a nonnegative function $\varphi \in C^\infty_0([0,1])$, $\varphi \not= 0$, and put
	\begin{equation*}\label{b_s}
 b_s :=  \|\varphi^{(s)}(y)\|_{L_p([0,1])}, \ s = 0,1,...,r.
\end{equation*}		
Define the functions $g$ and $h$ on $\RR$ by
\begin{equation*}\label{g}
g(x):= 
\begin{cases}
\varphi(\delta^{-1}(x - t_{i-1})), & \ \ x \in (t_{i-1}, t_i), \\
0, & \ \ \text{otherwise},	
\end{cases}	
\end{equation*}
and
	\begin{equation*}\label{h}
		h(x):= (gw^{-1})(x).		
	\end{equation*}
Let us estimate the norm $\norm{h}{W^r_{p,w}(\RR)}$. For a given $k \in \NN_0$ with $0 \le k \le r$,  we have
	\begin{equation}\label{h^(s)}
	h^{(k)} = (gw^{-1})^{(k)} = \sum_{s=0}^k \binom{k}{s} g^{(k-s)}(w^{-1})^{(s)}. 		
\end{equation}
By a direct computation we find that   for $x \in \RR$,
	\begin{equation}\label{w^(s)}
(w^{-1})^{(s)}(x) = (w^{-1})(x) (\sign (x))^s\sum_{j=1}^s c_{s,j}(\lambda,a) |x|^{\lambda_{s,j}},
\end{equation}
where  $\sign (x):= 1$ if $x \ge 0$, and $\sign (x):= -1$ if  $x < 0$, 
\begin{equation}\label{lambda_{s,s}}
\lambda_{s,s} = s(\lambda - 1) > \lambda_{s,s-1} > \cdots > \lambda_{s,1} = \lambda - s,
\end{equation}
 and  $c_{s,j}(\lambda,a)$ are polynomials in the variables $\lambda$ and $a$ of degree at most $s$ with respect to each variable.
Hence, we obtain
	\begin{equation}\label{h^(s)w(x)}
	h^{(k)}(x) w(x)= \sum_{s=0}^k \binom{k}{s} g^{(k-s)}(x) (\sign (x))^s\sum_{j=1}^s c_{s,j}(\lambda, a) |x|^{\lambda_{s,j}} 		
\end{equation}
which implies that 
	\begin{equation}\nonumber
	\int_{\RR}|h^{(k)} w|^p(x) \rd x 
	\le  C \max_{0\le s \le k} \ \max_{1 \le j \le s}  \int_{t_{i-1}}^{t_i} |x|^{p\lambda_{s,j}}|g^{(k-s)}(x)|^p \rd x.		
\end{equation}
From \eqref{lambda_{s,s}},  the inequality $n^{1/\lambda} \le x \le (2n + 2)n^{1/\lambda - 1}$ for $x \in [t_{i-1}, t_i]$, and 
\begin{equation*}\label{int-g^(k-s)}
  \int_{t_{i-1}}^{t_i} |g^{(k-s)}(x)|^p \rd x = b_{k-s}^p \delta^{-p(k-s-1/p)} 
  = 	b_{k-s}^p n^{p(k-s-1/p)(1 - 1/\lambda)},
\end{equation*}
we derive
	\begin{equation*}\label{int-h^(s)w}
		\begin{aligned}
	\int_{\RR}|h^{(k)} w|^p(x) \rd x 
	&\le  C \max_{0\le s \le k}   \int_{t_{i-1}}^{t_i} |x|^{p\lambda_{s,s}}|g^{(k-s)}(x)|^p \rd x
	\\
	& \le  C \max_{0\le s \le k}   \brac{n^{1/\lambda}}^{ps(\lambda - 1)}
	\int_{t_{i-1}}^{t_i}|g^{(k-s)}(x)|^p \rd x
	\\& 
	\le  C \max_{0\le s \le k} n^{ps(\lambda - 1)/\lambda}		 n^{p(k-s-1/p)(1 - 1/\lambda)}
	\\
	&= C n^{p(1-1/\lambda)(k-1/p)} \le  C n^{p(1-1/\lambda)(r-1/p)}.		
	\end{aligned}
\end{equation*}
If we define 
	\begin{equation*}\label{h-bar}
	\bar{h}:=  C^{-1} n^{-(1-1/\lambda)(r-1/p)}h,
\end{equation*}
then 
$\bar{h} \in \bW^r_{p,w}(\RR)$, $\supp (\bar{h}) \subset (t_{i-1},t_i)$ and
	\begin{equation*}\label{int-h^(s)w2}
	\begin{aligned}	
		\int_{\RR}|(\bar{h}w)(x)|^q \rd x 
	&=  C^{-q} n^{-q(1-1/\lambda)(r-1/p)}\int_{t_{i-1}}^{t_i}|g(x)|^q \rd x
	\\
	&
	=  C^{-q} n^{-q(1-1/\lambda)(r-1/p)} \delta \int_0^1|\varphi(x)|^q \rd x  \gg n^{-q(1-1/\lambda)(r-1/p+1/q)}
	\end{aligned}
\end{equation*}
 Since the interval $(t_{i-1}, t_i)$ does not contain any point from the set $\brab{m_1,...,m_n}$, we have $\bar{h}(m_k) = 0$, $k = 1,...,n$. Hence, by the inequality \eqref{rho_n>},
	\begin{equation}\nonumber
		\varrho_n\big(\bW^r_{p,w}(\RR), L_{q,w}(\RR)\big) 	\ge  \|\bar{h}\|_{L_{p,w}(\RR)}
	 \gg  n^{- r_{\lambda,p,q}}.
\end{equation}
The lower bound in \eqref{rho_n-d=1} in the case $p < q$ is proven.

We now prove the lower bound in \eqref{rho_n-d=1} in the case $p \ge q$.   By  \eqref{rho_n > d_n} we have
	\begin{equation}\label{rho_n>d_n}
	\varrho_n\big(\bW^r_{p,w}(\RR), L_{q,w}(\RR)\big) 	
	\ge 
	d_n\big(\bW^r_{p,w}(\RR), L_{q,w}(\RR)\big),
\end{equation}	
hence  it is sufficient to show that for $p \ge q$,
\begin{equation}\label{d_n>}
d_n\big(\bW^r_{p,w}(\RR), L_{q,w}(\RR)\big) 	
\gg  n^{- r_{\lambda,p,q}}.
\end{equation}
From Lemma \ref{lemma:B-NInequality}(i) we derive the inequality
	\begin{equation}\nonumber
	\|\varphi\|_{W^r_{p,w}(\RR)} 
	\ll
	m^{r_\lambda}\|\varphi\|_{L_{p,w}(\RR)}	 	 
	\ \ \ \
	\forall  \varphi \in \Pp_{m+1}, \ \ \forall m \in \NN.  
\end{equation}
This implies the inclusion 
$$
Cm^{-r_\lambda}\Bb^*_{N(m),p} \subset \bW^r_{p,w}(\RR)
$$ 
for some constant $C >0$ independent of $m$, where $N(m):=2j(m)$ and
$$
\Bb^*_{N(m),p}:= \brab{\varphi \in \Pp^*_{m+1}: \|\varphi\|_{L_{p,w}(\RR)} \le 1}
$$
is the unit ball of $N(m)$-dimensional subspace $L^*_{N(m),p}:= \Pp^*_{m+1}\cap L_{p,w}(\RR)$ in $L_{p,w}(\RR)$.
Given $n \in \NN$, let $N:=N(\bar{m}):= \min\brab{N(m): \, N(m) \ge 2n}$. Due to \eqref{<j(m)<} we have
\begin{equation} \label{N><}
2n\le N \asymp \bar{m} \asymp n.
\end{equation} 
Hence,
\begin{equation}\label{d_n-ge}
	d_n\big(\bW^r_{p,w}(\RR), L_{q,w}(\RR)\big) 	
	\ge 	d_n\big(C \bar{m}^{-r_\lambda} \Bb^*_{N,p}, L^*_{N,q}\big) 
	\asymp 	n^{-r_\lambda}d_n \big(\Bb^*_{N,p}, L^*_{N,q}\big).
\end{equation}
For $1\le \nu < \infty$ and $M \in \NN$, denote by $\ell^M_\nu$ the normed space of $\bx=(x_j)_{j=1}^M$ with $x_j \in \RR$, equipped with the norm
$$
\norm{\bx}{\ell^M_\nu}:= \brac{\sum_{j=1}^M |x_j|^\nu}^{1/\nu},
$$
and by $B^M_\nu$ the unit ball in $\ell^M_\nu$.
From Lemma \ref{lemma:MarcinkiewiczInequality} it follows that for $\nu = p,q$, 
 there holds the norm  equivalence  
\begin{equation}\label{}
	\|\varphi\|_{L_{w}^\nu(\RR)} 
	\asymp
	m^{-(1 - 1/\lambda)/\nu}	\|\bvarphi\|_{\ell^N_\nu} \ \ 
	\varphi \in \Pp^*_{m+1}, \ \ \forall m \in \NN, 
\end{equation}
where  $\bvarphi:=\big(\varphi(x_{m,k}) w(x_{m,k})\big)_{|k|\le j(\bar{m})}$.
Hence by \eqref{N><} and the well-known equality 
\begin{equation}\label{d_n=}
	d_n\big(B^N_p, \ell^N_q\big) 	
	=
	(N-n)^{1/q - 1/p},  \ \ N > n \ (p \ge q)
\end{equation}
 (see, e.g., \cite[Page 232]{Tikh1976}) we  deduce that
\begin{equation}\nonumber
	\begin{split}
	d_n \big(\Bb^*_{N,p}, L^*_{N,q}\big)
	&\asymp 
		n^{(1 - 1/\lambda)(1/p -  1/q)}	d_n\big(B^N_p, \ell^N_q\big)\\
		&= n^{(1 - 1/\lambda)(1/p -  1/q)}(N-n)^{1/q - 1/p} \asymp 
		n^{\delta_{\lambda,p,q}}.
	\end{split}
\end{equation}
This together with \eqref{rho_n>d_n}, \eqref{d_n>} and \eqref{d_n-ge} proves the lower bound in \eqref{rho_n-d=1}  in the case $p \ge q$. 
	\hfill
\end{proof}


By using a similar  technique with some modification, we now give upper and lower bounds for $\varrho_n\big(\bW^r_{p,w}(\RR),  L_{q,w}(\RR)\big)$  in the case when $p = \infty$ and $1\le q \le \infty$ and when the generating weights $w$ and $v$ are of the form:
\begin{equation} \label{w(x),p=infty}
	w(x):= 	\exp \brac{- a|x|^\lambda + b},
	\ \ \lambda > 1, \ a > 0, \ b \in \RR,
\end{equation} 
and 
\begin{equation} \label{v(x),p=infty}
	v(x):= 	 w^2(x) = \exp \brac{- 2a|x|^\lambda + 2b}, 
\end{equation}
where $\lambda , a, b$ are the same as in  \eqref{w(x),p=infty}. 

\begin{lemma} \label{lemma:Delta_k-p=infty }
	Let  $1\le q \le \infty$ and $r_{\lambda,\infty,q} >0$. Let the generating weights $w$ and $v$ be given by \eqref{w(x),p=infty} and \eqref{v(x),p=infty}.  Then we have that
	for every $ f \in W^r_{\infty,w}(\RR)$,
	\begin{equation}\label{Series2-p=infty}
		f
		= 	\sum_{k \in \NN_0}\Delta_k(f)
	\end{equation}
	with absolute convergence in the space $L_{q,w}(\RR)$  of the series. 
	Moreover,
	\begin{equation}\label{Delta_kf-p=infty}
		\big\|\Delta_k (f)\big\|_{L_{q,w}(\RR)}  
		\ll 2^{-  r_{\lambda,\infty,q} k} k\,\|f\|_{W^r_{\infty,w}(\RR)}, 
		\ \  k \in \NN_0.
	\end{equation}
\end{lemma}

\begin{proof}
	We first prove the inequality
		\begin{equation}\label{|f - I_m f|<-p=infty}
		\|f - I_m (f)\|_{L_{w}^\infty(\RR)}  
		\ll
			m^{-r_\lambda} \log m \,\|f \|_{W^r_{\infty,w}(\RR)}, \ \  f \in W^r_{\infty,w}(\RR).  
	\end{equation}
For $1 \le p \le \infty$ and $f \in L_{p,w}(\RR)$, we define 	
		\begin{equation}\label{}
	E_m(f)_{w,p}:= \inf_{\varphi \in \Pp_m}	\|f - \varphi\|_{L_{p,w}(\RR)}.  
	\end{equation}
	Then there holds the inequality \cite{DM2003}
			\begin{equation}\label{}
		E_m(f)_{w,p} 
		\ll 
	m^{-r_\lambda}  \|f \|_{W^r_{p,w}(\RR)}, \ \  f \in W^r_{p,w}(\RR).  
	\end{equation}
	On the other hand, by \cite[Theorem 3.2]{MN2010}
		\begin{equation}\label{|f - I_m f|<-p=infty2}
	\|f - I_m (f)\|_{L_{w}^\infty(\RR)}  
	\ll
	\log m \, 	E_{m'}(f)_{w,\infty} + e^{-cm}\|f \|_{L_{w}^\infty(\RR)},
\end{equation}
where $m'= \left\lfloor \brac{\frac{\rho}{\rho + 1}}^\lambda m \right\rfloor \asymp m$	and $c$ is a positive constant independent of $m$ and $f$.
From the last two inequalities we easily deduce \eqref{|f - I_m f|<-p=infty}. Now by using \eqref{|f - I_m f|<-p=infty}, we can prove the lemma in the same way as the proof of Lemma \ref{lemma:Delta_k}.
	\hfill	
\end{proof}	
\begin{theorem} \label{thm:I_m,p=infty}
	Let  $1\le q \le \infty$ and $r_{\lambda,\infty,q} >0$. Let the generating weights $w$ and $v$ be given by \eqref{w(x),p=infty} and \eqref{v(x),p=infty}.  For any $n \in \NN$, let $k(n)$ be the largest integer such that $2^{k(n)}  \le n$. Then we have 
	\begin{equation}\label{rho_n-d=1,p=infty}
			n^{- r_{\lambda,\infty,q}}
			\ll
		\varrho_n\big(\bW^r_{\infty,w}(\RR),  L_{q,w}(\RR)\big) 
		\le
		\sup_{f\in \bW^r_{\infty,w}(\RR)} \big\|f - I_{m_{k(n)}}(f)\big\|_{ L_{q,w}(\RR)} 
		\ll
		n^{- r_{\lambda,\infty,q}} \log n.
	\end{equation}
\end{theorem}
\begin{proof}
The proof of the upper bound in \eqref{rho_n-d=1,p=infty} is similar to the proof of the upper bound in \eqref{rho_n-d=1} of Theorem \ref{thm:I_m} with replacing Lemma \ref{lemma:Delta_k} by Lemma \ref{lemma:Delta_k-p=infty }. The lower bound in \eqref{rho_n-d=1,p=infty} can be proven in the same way as the proof of  the lower bound \eqref{rho_n-d=1} of Theorem~\ref{thm:I_m} in the case $p \ge q$.
	\hfill	
\end{proof}



	\section{Sparse-grid sampling recovery in high dimension}
\label{Multivariate sampling recovery}

In this section, in the high dimensional case ( $d \ge 2$), we establish upper and lower bounds of   $\varrho_n\big(\bW^r_{p,w}(\RRd), L_{q,w}(\RRd)\big)$ and construct linear sampling  algorithms based on  step-hyperbolic-cross grids of sparse sampled points which realize the upper bounds. 

For $\bx \in \RRd$ and $e \subset \brab{1,...,d}$, let $\bx^e \in \RR^{|e|}$ be defined by $(x^e)_i := x_i$, and  $\bar{\bx}^e\in \RR^{d-|e|}$ by $(\bar{x}^e)_i := x_i$, $i \in \brab{1,...,d} \setminus e$. With an abuse we write 
$(\bx^e,\bar{\bx}^e) = \bx$.
For the proof of the following lemma, see \cite[Lemma 3.2]{DD2023}. 

\begin{lemma} \label{lemma:g(bx^e}
	Let $1\le p \le \infty$,  $e \subset \brab{1,...,d}$ and $\br \in \NNd_0$. Assume that $f$ is a function on $\RRd$  such that for every $\bk \le \br$, $D^\bk f \in L_{p,w}(\RRd)$. 
	Put for  $\bk \le \br$ and $\bar{\bx}^e \in \RR^{d-|e|}$,
	\begin{equation*}\label{g(bx^e}
	g(\bx^e): =  D^{\bar{\bk}^e} f(\bx^e,\bar{\bx}^e).
	\end{equation*}
Then $D^\bs g \in L_{p,w}(\RR^{|e|})$ for every $\bs \le \bk^{e}$ and almost every 
$\bar{\bx}^e \in \RR^{d-|e|}$.
\end{lemma}

%

Based on the operators $\Delta_k, \ k \in \NN_0$, defined in \eqref{Delta}, we construct sampling operators on $\RRd$ by using the well-known Smolyak algorithm. For convenience, with an abuse we make use of the notation:
$$
S_{2^k}:= I_{m_k} \in \Ss_{2^k}, \ \  k \in \NN_0,
$$
where we recall that the sampling operator $I_m$ is given as in \eqref{I_m f}, and  $m_k$ is the largest number such that $m_k + 1 \le 2^k$.  
Then we have
$$
\Delta_k =  S_{2^k} - S_{2^{k-1}}, \  \ k \in \NN_0.
$$
We also define for $k \in \NN_0$, the one-dimensional operators
\begin{equation*}\label{E}
	E_k f:= f - S_{2^k} (f), \ \ \ k \in \NN_0.
\end{equation*}
For $\bk \in \NNd$, the $d$-dimensional operators $S_{2^\bk}$,  $\Delta_\bk$ and $E_\bk$ are defined as the tensor  product of one-dimensional operators:
\begin{equation}\label{tensor-product}
	S_{2^\bk}:= \bigotimes_{i=1}^d S_{2^{k_i}} , \ \	
	\Delta_\bk:= \bigotimes_{i=1}^d \Delta_{k_i}, \ \ 	
	E_\bk:= \bigotimes_{i=1}^d E_{k_i}, 
\end{equation}
where $2^\bk:= (2^{k_1},\cdots, 2^{k_d})$ and 
the univariate operators $S_{2^{k_j}}$, $\Delta_{k_j}$ and $E_{k_j}$ 
 are successively applied to the univariate functions $\bigotimes_{i<j} S_{2^{k_i}}(f)$, $\bigotimes_{i<j} \Delta_{k_i}(f)$ and $\bigotimes_{i<j} E_{k_i} $, respectively, by considering them  as 
functions of  variable $x_j$ with the other variables held fixed. The operators $S_{2^\bk}$, $\Delta_\bk$ and $E_\bk$ are well-defined for continuous functions on $\RRd$, in particular for ones from $W^r_{p,w}(\RRd)$.

Notice that  $S_{2^\bk}$ is a  sampling operator on $\RRd$  given by
\begin{equation}\label{I_2^bkf}
	S_{2^\bk}f = \sum_{\bs=\bone}^{m_\bk} f(\bx_{m_\bk,\bs}) \varphi_{m_\bk,\bs} , 
	\ \ \{\bx_{m_\bk,\bs}\}_{\bone \le \bs \le m_\bk}\subset \RRd,
\end{equation}
where $$
m_\bk:= \brac{m_{k_1},...,m_{k_d}}, \ \ \ 
\bx_{m_\bk,\bs}:= \brac{x_{m_{k_1},s_1},...,x_{m_{k_d},s_d}}, \ \ \ 
\varphi_{m_\bk,\bs}:= \bigotimes_{i=1}^d \varphi_{m_{k_i},s_i} ,
$$
 and the summation $\sum_{\bs=\bone}^{m_\bk}$ means that the sum is taken over all $\bs$ such that $\bone \le \bs \le m_\bk$. Hence we derive that 
\begin{equation}\label{Delta_bk}
\Delta_\bk f =  \sum_{e \subset \brab{1,...,d}} (-1)^{d - |e|}S_{2^{\bk(e)}} f
	= \sum_{e \subset \brab{1,...,d}} (-1)^{d - |e|}\sum_{\bs=\bone}^{m_{\bk(e)}} 
	f(\bx_{\bk(e),\bs}) \varphi_{\bk(e),\bs} , 
\end{equation}
where  $\bk(e) \in \NNd_0$ is defined by $k(e)_i = k_i$, $i \in e$, and 	$k(e)_i = \max(k_i-1,0)$, $i \not\in e$.  


\begin{lemma} \label{lemma:E_k}
		Let $1< p < \infty$, $1\le q \le \infty$ and $r_{\lambda,p,q} >0$.    Then we have that
	\begin{equation*}\label{E_k}
	\big\|E_\bk (f)\big\|_{L_{q,w}(\RRd)}  
		\leq C 2^{-r_{\lambda,p,q}|\bk|_1} \|f\|_{W^r_{p,w}(\RRd)}, 
		\ \  \bk \in \NNd_0, \ \ f \in W^r_{p,w}(\RRd).
	\end{equation*}
\end{lemma}

\begin{proof} 
	The proof of this lemma is similar to the proof of Lemma 3.2  in \cite{DD2024}. For completeness let us process it.
	The case $d=1$ of the lemma follows from Lemma \ref{lemma:I_m}.  For simplicity we prove the lemma for the case $d=2$ and $q < \infty$. The general case can be proven in the same way by induction on $d$.
	Indeed, by  applying successively the case  $d=1$ of the lemma with respect to variables $x_2$ and $x_1$ we obtain 	
	\begin{equation}\nonumber
		\begin{aligned}
			\big\|E_{(k_1,k_2)}f\big\|_{L_{q,w}(\RR^2)}^q
			& =
			\int_{\RR}\int_{\RR}\big|E_{k_2} (E_{k_1}f(x_1,x_2))\big|^q w(\bx)^q\rd x_2 \rd x_1
			\\
			&\leq 
			2^{-q r_{\lambda,p,q} k_2}	\int_{\RR} \sum_{s_2=0}^r\int_{\RR}
			\big|D^{(0,s_2)} (E_{k_1}f(x_1,x_2))\big|^q w(\bx)^q\rd x_2 \rd x_1
			\\
			&= 
			2^{-q r_{\lambda,p,q} k_2}	\int_{\RR} \sum_{s_2=0}^r\int_{\RR}
			\big| (E_{k_1}D^{(0,s_2)}f(x_1,x_2))\big|^q w(\bx)^q\rd x_2 \rd x_1
			\\
			&\le
			2^{-q r_{\lambda,p,q} k_2}	\int_{\RR} \sum_{s_2=0}^r\int_{\RR}
			2^{-q r_{\lambda,p,q} k_1}	\sum_{s_1=0}^r\int_{\RR}
			\big| D^{(s_1,s_2)}f(x_1,x_2)\big|^q w(\bx)^q\rd x_1 \rd x_2
			\\
			&= 
			2^{-q r_{\lambda,p,q} |\bk|_1}\sum_{|\bs|_\infty \le r}\int_{\RR^2}
			\big| D^{(s)}f(\bx)\big|^q w(\bx)^q\rd \bx
			\\
			&= 
			2^{- q r_{\lambda,p,q}|\bk|_1} \|f\|_{W^r_{p,w}(\RR^2)}^q.   
		\end{aligned}
	\end{equation}
	
		\hfill
\end{proof}

We say that $\bk \to \infty$, $\bk \in \NNd_0$, if and only if $k_i \to \infty$ for every $i = 1,...,d$.
\begin{lemma} \label{lemma:Delta_k}
	Let $1< p < \infty$, $1\le q \le \infty$ and $r_{\lambda,p,q} >0$.    Then we have that
	for every $ f \in W^r_{p,w}(\RRd)$,
		\begin{equation}\label{Series1}
f
	= 	\sum_{\bk \in \NNd_0}\Delta_\bk (f)
	\end{equation}
with absolute convergence in the space $L_{q,w}(\RRd)$  of the series, and
	\begin{equation}\label{Delta_k}
			\big\|\Delta_\bk (f)\big\|_{L_{q,w}(\RRd)}  
		\leq C 2^{-r_{\lambda,p,q}|\bk|_1} \|f\|_{W^r_{p,w}(\RRd)}, 
		\ \  \bk \in \NNd_0.
	\end{equation}
\end{lemma}

\begin{proof}  The operator $\Delta_\bk$ can be represented in the form
	\begin{equation}\nonumber
	\Delta_\bk (f)
	= 	\sum_{e \subset \brab{1,...,d}} (-1)^{|e|} E_{\bk(e)} (f).
\end{equation}	
Therefore, by using Lemma \ref{lemma:E_k} we derive that for every $f \in W^r_{p,w}(\RRd)$ and  $\bk \in \NNd_0$,
	\begin{equation}\nonumber
		\begin{aligned}
			\big\|\Delta_\bk f\big\|_{L_{q,w}(\RRd)} 
			& \le
			\sum_{e \subset \brab{1,...,d}} \big\|E_{\bk(e)} f\big\|_{L_{q,w}(\RRd)}
			\\
			&\leq 
			\sum_{e \subset \brab{1,...,d}}  C 2^{-r_{\lambda,p,q}|\bk(e)|_1} \|f\|_{W^r_{p,w}(\RRd)}
			\le C 2^{-r_{\lambda,p,q}|\bk|_1} \|f\|_{W^r_{p,w}(\RRd)}
		\end{aligned}
	\end{equation}
which proves \eqref{Delta_k} and hence  the absolute convergence of the series in \eqref{Series1} follows.
 Notice that 
\begin{equation}\nonumber
f - S_{2^\bk}f
	= 	\sum_{e \subset \brab{1,...,d}, \ e \not= \varnothing} (-1)^{|e|} E_{\bk^e}f,
\end{equation}	
where recall $\bk^e \in \NNd_0$ is defined by $k^e_i = k_i$, $i \in e$, and 	$k^e_i = 0$, $i \not\in e$.
By using Lemma~\ref{lemma:E_k} we derive  for $\bk \in \NNd_0$ and $f \in W^r_{p,w}(\RRd)$,
\begin{equation}\nonumber
	\begin{aligned}
		\big\|f - S_{2^\bk}f\big\|_{L_{q,w}(\RRd)}  
		& \le
		\sum_{e \subset \brab{1,...,d}, \ e \not= \varnothing} \big\|E_{\bk^e} f\big\|_{L_{q,w}(\RRd)}
		\\
		&\leq 
		C \max_{e \subset \brab{1,...,d}, \ e \not= \varnothing} \ \max_{1 \le i \le d} 2^{-r_{\lambda,p,q}|k^e_i|} \|f\|_{W^r_{p,w}(\RRd)}
	\\
	&\leq 
	C \max_{1 \le i \le d} 2^{-r_{\lambda,p,q}|k_i|} \|f\|_{W^r_{p,w}(\RRd)},
	\end{aligned}
\end{equation}
which is going to $0$ when $\bk \to \infty$. This together with the obvious equality
	\begin{equation}\nonumber
S_{2^\bk}
	= 	\sum_{\bs \le \bk} \Delta_{\bs}
\end{equation}	
proves  \eqref{Delta_k}.
	\hfill
\end{proof}

We now define an  algorithm for sampling on sparse grids   initiated by Smolyak (for detail see \cite[Sections 4.2 and 5.3]{DTU18B}). For $m \in \NN_0$, we define the operator
	\begin{equation}\label{P_m}
	P_m
	:= 	\sum_{|\bk|_1 \le m } \Delta_{\bk}.
\end{equation}	
From \eqref{Delta_bk} we can see that $P_m$ is a linear sampling algorithm on $\RRd$ of the form \eqref{S_k} :
\begin{equation}\label{I_xi}
	P_m (f)
	= 	\sum_{|\bk|_1 \le m } \ \sum_{e \subset \brab{1,...,d}} (-1)^{d - |e|}\ \sum_{\bs=\bone}^{2^{\bk(e)}}  f(\bx_{\bk(e),\bs}) \phi_{\bk(e),\bs}
	= \sum_{(\bk,e,\bs) \in G(m)} f(\bx_{\bk,e,\bs}) \phi_{\bk,e,\bs} , 
\end{equation}
where 
$$
\bx_{\bk,e,\bs}:= \bx_{\bk(e),\bs}, \quad \phi_{\bk,e,\bs}:= (-1)^{d - |e|}\phi_{\bk(e),\bs}
$$ 
and 
\begin{equation}\nonumber
G(m)	:= \brab{(\bk,e,\bs): \ |\bk|_1 \le m, \,   e \subset \brab{1,...,d}, \,  \bone \le \bs \le \bk(e)}
\end{equation}
is a finite set.
The set of  sampled points in this operator  
\begin{equation}\nonumber
H(m):=\brab{\bx_{\bk,e,\bs}}_{(\bk,e,\bs) \in G(m)}
\end{equation}
is a step hyperbolic cross   in the function domain $\RRd$. 
The number of sampled in the operator $P_m$ is 
\begin{equation}\nonumber
		|H(m)|
	= |G(m)|
	= 	\sum_{|\bk|_1 \le m } \ \sum_{e \subset \brab{1,...,d}}2^{|\bk(e)|_1}
\end{equation}
which can be estimated as 
\begin{equation}\label{|G(xi)|}
	|H(m)|
	\asymp	\sum_{|\bk|_1 \le m }2^{|\bk|_1} \ \asymp \ 2^m m^{d - 1}, \ \ m \ge 1.
\end{equation}
The sampling operator  $P_m$ plays a basic  role in the proof of the upper bound \eqref{UpperB-introduction}.

\begin{lemma} \label{lemma:I_xi-error}
	Let $1< p < \infty$, $1\le q \le \infty$ and $r_{\lambda,p,q} >0$.    Then we have 	for every $m > 1$  and every  $f \in W^r_{p,w}(\RRd)$,
\begin{equation}\label{upperbound}
	\big\|f - P_m (f)\big\|_{L_{q,w}(\RRd)}  
	\ll 
	 \|f\|_{W^r_{p,w}(\RRd)}
\begin{cases}
	2^{- r_\lambda} m^{d - 1}  &  \ \ \text{if} \ \ p = q, \\
	2^{- r_{\lambda,p,q}m} m^{(d - 1)/q} &  \ \ \text{if} \ \ p \not= q< \infty, \\
2^{- r_{\lambda,p,q}m} m^{(d - 1)} &  \ \ \text{if} \ \ p \not= q=\infty.
\end{cases}		
\end{equation}
\end{lemma}

\begin{proof}  From  Lemma \ref{lemma:Delta_k} we derive that for $m \ge 1$ and $f \in W^r_{p,w}(\RRd)$,
		\begin{equation}\label{Series1}
	f - P_m (f) 
	= 	\sum_{|\bk|_1 > m } \Delta_\bk f, \ \  \Delta_\bk f \in \Pp_{2^{\bk}},
\end{equation}
with absolute convergence in the space $L_{q,w}(\RRd)$  of the series, and there holds  \eqref{Delta_k}. If $p \not= q$, applying Lemma	\ref{lemma:IneqL_qNorm<L_pNorm} in Appendix and \eqref{Delta_k}, we obtain \eqref{upperbound}:
	\begin{equation}\nonumber
		\begin{aligned}
		\big\|f - P_m (f) \big\|_{L_{q,w}(\RRd)}^q  
			& \ll
				\sum_{|\bk|_1 > m } \big\|2^{\delta_{\lambda,p,q}|\bk|_1}\Delta_\bk f \big\|_{L_{p,w}(\RRd)}^q
				\ll \sum_{|\bk|_1 > m } 2^{-q r_{\lambda,p,q}|\bk|_1}  \|f\|_{W^r_{p,w}(\RRd)}^q
		\\
		&= 
	 \|f\|_{W^r_{p,w}(\RRd)}^q \sum_{|\bk|_1 > m } 2^{-q r_{\lambda,p,q}|\bk|_1} 
	\ll  2^{- q r_{\lambda,p,q}m} m^{d - 1} \|f\|_{W^r_{p,w}(\RRd)}^q.
		\end{aligned}
	\end{equation}
	
	If $p  = q$  or $p\not=q=\infty $, the upper bound \eqref{upperbound} can be derived similarly by using \eqref{Series1}, \eqref{Delta_k} and the inequality
	\begin{equation}\nonumber
			\big\|f - P_m (f) \big\|_{L_{q,w}(\RRd)} 
			\le
			\sum_{|\bk|_1 > m } \big\|\Delta_\bk f \big\|_{L_{q,w}(\RRd)}.			
	\end{equation}
	\hfill
\end{proof}


\begin{theorem} \label{thm:rho_n(W)}
	Let  $1< p < \infty$, $1\le q \le \infty$ and $r_{\lambda,p,q} >0$, and  
	denote 
	$$
	\varrho_n:= \varrho_n(\bW^r_{p,w}(\RRd), L_{q,w}(\RRd)) ).
	$$  
For every $n \in \NN$, let $m_n$ be the largest number such that $|G(m_n)| \le n$.	  Then  $P_{m_n}$ defines a sampling operator belonging to $\Ss_n$, and we have that
	\begin{equation}\label{rho_n(W)}
	\varrho_n
	\le 
	\sup_{f\in \bW^r_{p,w}(\RRd)}	
	\big\|f - P_{m_n} f\big\|_{L_{q,w}(\RRd)}   
	\ll 
		\begin{cases}
		n^{-r_{\lambda}} (\log n)^{(r_{\lambda} + 1)(d-1)} &   \ \text{if} \ \ p = q, \\
		n^{-r_{\lambda,p,q}} (\log n)^{(r_{\lambda,p,q} + 1/q)(d-1)} &   \ \text{if} \ \ p \not= q<\infty,
		\\
		n^{-r_{\lambda,p,q}} (\log n)^{(r_{\lambda,p,q} + 1)(d-1)} &   \ \text{if} \ \ p \not= q=\infty;
	\end{cases}	
	\end{equation}
	and
	 \begin{equation}\label{LowerB}
		\varrho_n
		\gg
		\begin{cases}
				n^{-r_{\lambda}} (\log n)^{r_{\lambda}(d-1)} &   \ \text{if} \ \ p = q, \\
		n^{-r_{\lambda,p,q}} (\log n)^{r_{\lambda,p,q}(d-1)} &   \ \text{if} \ \ p < q, \\
		n^{-r_{\lambda,p,q}} &  \ \ \text{if} \ \ p > q.
		\end{cases}	 
	\end{equation}
\end{theorem}

\begin{proof}  
 Let us prove the case $p \not= q < \infty$ of the upper bounds \eqref{rho_n(W)}. The cases $p = q$ and $p\not=q=\infty $ can be proven in a similar manner.
 From \eqref{|G(xi)|} it follows 
$$
\ 2^{m_n} m_n^{d - 1} \asymp |G(m_n)|  \asymp n.
$$ 
Hence we deduce the asymptotic equivalences
$$
\ 2^{-m_n}  \asymp n^{-1} (\log n)^{d-1}, \ \  m_n \asymp \log n,
$$
which together with Lemma \ref{lemma:I_xi-error}  yield that
	\begin{equation*}\label{I_xi-error}
		\begin{aligned}
		\varrho_n
		&\le 
	\sup_{f\in \bW^r_{p,w}(\RRd)}	
	\big\|f - P_{m_n} f\big\|_{L_{q,w}(\RRd)}   
		\\
		&
		\leq 
		C 2^{- r_\lambda m_n} m_n^{(d-1)/q}
		\asymp  	n^{-r_{\lambda,p,q}} (\log n)^{(r_{\lambda,p,q} + 1/q)(d-1)}.
		\end{aligned}
	\end{equation*}
The upper bound in \eqref{rho_n(W)} for the case $p \not= q<\infty$ is proven.

We now prove the lower bounds  in \eqref{LowerB}.  The lower bound for the case $p > q$ can be derived from the lower bound in Theorem \ref{thm:I_m} for $d=1$.
We prove it for the cases $p=q$ and  $p < q$ merged as the case $p \le q$,  by using the inequality \eqref{rho_n>}.
For $M \ge 1$, we define the set 
	\begin{equation*}\label{Gamma-set}
\Gamma_d(M):= \brab{\bs \in \NNd: \, \prod_{i=1}^d s_i \le 2M, \ s_i \ge M^{1/d}, \ i=1,...,d}.
\end{equation*}
Then we have by  \cite[(3.15)]{DD2023}
	\begin{equation}\label{|Gamma|}
	|\Gamma_d(M)| \asymp M (\log M)^{d-1}, \ \ M>1.
\end{equation}

For a given $n \in \NN$, let $\brab{\bxi_1,...,\bxi_n} \subset \RRd$ be arbitrary  $n$ points. Denote by $M_n$  the smallest number such that $|\Gamma_d(M_n)| \ge n + 1$. We define the $d$-parallelepiped $K_\bs$ for $\bs \in \NNd_0$ of size 
$$
\delta:= M_n^{\frac{1/\lambda - 1}{d}}
$$
 by
	\begin{equation*}\label{K_bs}
K_\bs:= \prod_{i=1}^d K_{s_i}, \ \ K_{s_i}:= (\delta s_i, \delta s_{i-1}).
\end{equation*}
 Since  $|\Gamma_d(M_n)| > n$, there exists a multi-index 
$\bs \in \Gamma_d(M_n)$ such that $K_\bs$ does not contain any point from $\brab{\bxi_1,...,\bxi_n}$.

As in the proof of Theorem \ref{thm:I_m}, we take a nonnegative function $\varphi \in C^\infty_0([0,1])$, $\varphi \not= 0$, and put
	\begin{equation}\label{b_s-d}
	b_s :=  \|\varphi^{(s)}(y)\|_{L_p([0,1])}, \ s = 0,1,...,r.
\end{equation}				
For $i= 1,...,d$, we define the univariate functions $g_i$ in variable $x_i$ by
\begin{equation}\label{g_i}
	g_i(x_i):= 
	\begin{cases}
		\varphi(\delta^{-1}(x_i - \delta s_{i-1})), & \ \ x_i  \in K_{s_i}, \\
		0, & \ \ \text{otherwise}.	
	\end{cases}	
\end{equation}
Then the multivariate functions $g$ and $h$ on $\RRd$ are defined by
	\begin{equation*}\label{g(bx)}
	g(\bx):= \prod_{i=1}^d g_i(x_i),
\end{equation*}
and  
\begin{equation}\label{h(bx)}
	h(\bx):= (gw^{-1})(\bx)= \prod_{i=1}^d g_i(x_i)w^{-1}(x_i)=:
	\prod_{i=1}^d h_i(x_i).		
\end{equation}
Let us estimate the norm $\norm{h}{W^r_{p,w}(\RRd)}$. For every $\bk \in \NNd_0$ with 
$0 \le |\bk|_\infty \le r$,  we prove the inequality
\begin{equation}\label{int-D^br}
	\int_{\RRd}\big|(D^{\bk} h) w\big|^p(\bx) \rd \bx 
	\le  C   M_n^{(1 - 1/\lambda)(r - 1/p)}.
\end{equation}
We have
\begin{equation}\label{D^bk h}
	D^\bk h  = \prod_{i=1}^d h_i^{(k_i)}. 		
\end{equation}
Similarly to \eqref{h^(s)}--\eqref{h^(s)w(x)} we derive that for every $i = 1,...,d$,
\begin{equation*}\label{h^(s-i)w(x_i)}
	h_i^{(k_i)}(x_i) w(x_i)= \sum_{\nu_i=0}^{k_i} \binom{k_i}{\nu_i} g_i^{(k_i- \nu_i)}(x_i) (\sign (x_i))^{\nu_i}\sum_{\eta_i=1}^{\nu_i} c_{\nu_i,\eta_i}(\lambda,a) |x_i|^{\lambda_{\nu_i,\eta_i}}, 		
\end{equation*}
where
\begin{equation*}\label{lambda_{nu,nu}}
	\lambda_{\nu_i,\nu_i} = \nu_i(\lambda - 1) > \lambda_{\nu_i,\nu_i-1} > \cdots > \lambda_{\nu_i,1} = \lambda - \nu_i,
\end{equation*}
and  $c_{\nu_i,\eta_i}(\lambda,a)$ are polynomials in the variables $\lambda$ and $a$ of degree at most $\nu_i$ with respect to each variable.
This together with \eqref{b_s-d}, \eqref{g_i} and the inequalities $s_i \ge M_n^{\frac{1}{d}}$ and $\lambda_{\nu_i,\nu_i} = \nu_i(\lambda - 1) \ge  0$  yields that 
\begin{equation}\label{int-h^(s)w-d>1}
	\begin{aligned}
			\int_{\RR}\big|h_i^{(k_i)}(x_i) w(x_i)\big|^p\rd x_i
	&\le  C \max_{0\le \nu_i \le k_i} \ \max_{1 \le \eta_i \le \nu_i}  
	\int_{K_{s_i}}|x_i|^{p\lambda_{\nu_i,\eta_i}}\big|g^{(k_i-\nu_i)}(x_i)\big|^p \rd x_i 
	\\
	&\le  C \max_{0\le \nu_i \le k_i} (\delta s_i)^{p\lambda_{\nu_i,\nu_i}}
	\int_{K_{s_i}}\big|g^{(k_i-\nu_i)}(x_i)\big|^p \rd x_i 
	\\
	&\le C \max_{0\le \nu_i \le k_i} (\delta s_i)^{p\nu_i(\lambda - 1)}
	\delta^{- p(k_i - \nu_i -1/p)} b_{k_i - \nu_i}^p
	\\
	&= C \delta^{- p(k_i -1/p)} \max_{0\le \nu_i \le k_i} \brac{\delta^\lambda  s_i^{\lambda -1}}^{p\nu_i}. 
\end{aligned}
\end{equation}
Since $s_i \ge M_n^{\frac{1}{d}}$ and $\delta:= M_n^{\frac{1/\lambda - 1}{d}}$, we have that
$\delta^\lambda  s_i^{\lambda -1} \ge 1$, and consequently, 
$$
\max_{0\le \nu_i \le k_i} \brac{\delta^\lambda  s_i^{\lambda -1}}^{p\nu_i}
 = 
 \brac{\delta^\lambda  s_i^{\lambda -1}}^{pk_i}.
$$
This equality,  the estimates \eqref{int-h^(s)w-d>1} and the inequalities  $0 \le k_i \le r$ and 
$\delta s_i \ge 1$ yield that
\begin{equation}\nonumber
	\begin{aligned}
		\int_{\RR}\big|h_i^{(k_i)}(x_i) w(x_i)\big|^p\rd x_i
		&\le  C \delta^{- p(k_i -1/p)} \brac{\delta^\lambda  s_i^{\lambda -1}}^{p k_i} 
		= C \delta \brac{\delta s_i}^{p k_i(\lambda -1)}
		\\ 
		& \le  C \delta \brac{\delta s_i}^{p r(\lambda -1)} 
		=  C \delta ^{p r(\lambda - 1) + 1} s_i^{p r(\lambda - 1)}. 
	\end{aligned}
\end{equation}
Hence, by \eqref{D^bk h} we deduce 
\begin{equation}\nonumber
	\begin{aligned}
		\int_{\RRd}|(D^\bk h) w|^p(\bx) \rd \bx 
		&= \prod_{i=1}^d  \int_{\RR}\big|h^{(k_i)}(x_i) w(x_i)\big|^p\rd x_i
		\\
		&  \le  C \prod_{i=1}^d  \delta ^{p r(\lambda - 1) + 1}  s_i^{p r(\lambda - 1)} 
		 \le  C \delta ^{d\brac{p(r(\lambda - 1) + 1}} \brac{\prod_{i=1}^d  s_i}^{p r(\lambda - 1)}.
\end{aligned}
\end{equation}
Since  $\prod_{i=1}^d  s_i \le 2M_n$, $\delta:= M_n^{\frac{1/\lambda - 1}{d}}$ and $\lambda  > 1$, we can continue the estimation as
\begin{equation*}\label{int-D^br2}
	\begin{aligned}
		\int_{\RRd}|(D^\bk h) w|^p(\bx) \rd \bx 
		\le  C  M_n^{p(r(\lambda - 1) + 1/p)(1/\lambda - 1)}   M_n^{p r(\lambda - 1)} 
			= C   M_n^{p(1 - 1/\lambda)(r - 1/p)},
	\end{aligned}
\end{equation*}
which completes the proof of the inequality \eqref{int-D^br}.
This  inequality means that $h \in W^r_{p,w}(\RRd)$ and  
$$
\norm{h}{W^r_{p,w}(\RRd)} \le  C   M_n^{(1 - 1/\lambda)(r - 1/p)}.
$$
 If  we define 
\begin{equation*}\label{h-bar}
	\bar{h}:=  C^{-1} M_n^{-(1-1/\lambda)(r-1/p)}h,
\end{equation*}
then $\bar{h}$ is nonnegative, $\bar{h} \in \bW^r_{p,w}(\RR)$, $\supp \bar{h} \subset K_\bs$ and by \eqref{b_s-d}--\eqref{h(bx)}, we have that for $q < \infty$,
\begin{equation}\nonumber
	\begin{aligned}	
		\int_{\RRd}|\bar{h}w|^q(\bx) \rd \bx 
		&=  C^{-q} M_n^{-q(1-1/\lambda)(r-1/p)}\int_{\RRd}|h w|^q(\bx) \rd \bx \\
		&
		 = C^{-q} M_n^{-q(1-1/\lambda)(r-1/p)}\prod_{i=1}^d \int_{K_{s_i}} |g_i(x_i)|^q \rd x_i \\
		&=  C^{-q} M_n^{-q(1-1/\lambda)(r-1/p)} \prod_{i=1}^d \delta \int_0^1|\varphi(x)|^q \rd x \\
		&=  {C'}^q M_n^{-q(1-1/\lambda)(r-1/p)}M_n^{1/\lambda - 1}
		=  {C'}^q M_n^{-q r_{\lambda,p,q}}.
	\end{aligned}
\end{equation}
From the definition of $M_n$ and \eqref{|Gamma|} it follows that
$$M_n(\log M_n)^{d-1} \asymp |\Gamma(M_n)| \asymp n,$$ 
 which implies that $M_n^{-1} \asymp n^{-1} (\log n)^{d-1}$. This allows to receive the estimate
\begin{equation}\label{int-h-bar2}
\|\bar{h}\|_{L_{q,w}(\RRd)}  
		=  C' M_n^{-r_{\lambda,p,q}}
		\gg
		n^{-r_{\lambda,p,q}} (\log n)^{r_{\lambda,p,q}(d-1)}.
\end{equation}
 Since the interval $K_\bs$ does not contain any point from the set $\brab{\bxi_1,...,\bxi_n}$ which has been arbitrarily chosen, we have 
 $$\bar{h}(\bxi_k) = 0, \ \ k = 1,...,n.
 $$
  Hence, by  the inequality \eqref{rho_n>} and \eqref{int-h-bar2} we have that
\begin{equation*}\label{int-h-bar3}
	\varrho_n
	\ge
\|\bar{h}\|_{L_{q,w}(\RRd)}  
	\gg 
		n^{-r_{\lambda,p,q}} (\log n)^{r_{\lambda,p,q}(d-1)}.
\end{equation*}
The lower bound in \eqref{rho_n(W)} for the case $p \le q< \infty$  is proven. It can be proven similarly for the case $p < q = \infty$ with a certain modification.
	\hfill
\end{proof}

We have proven  upper and lower bounds of the sampling widths 	$\varrho_n(\bW^r_{p,w}(\RRd), L_{q,w}(\RRd))$ for $1< p < \infty$ and $1\le q \le \infty$. In a similar way with a certain modification, we can prove upper and lower bounds of these sampling widths for $p=\infty$ and $1\le q \le \infty$ in the case when the generating weights $w$ and $v$ are of the form \eqref{w(x),p=infty} and \eqref{v(x),p=infty}.
More precisely, by using the same technique and by replacing \eqref{|f - I_m f|<} in  Lemma \ref{lemma:I_m} with \eqref{|f - I_m f|<-p=infty} we derive the following result.

\begin{lemma} \label{lemma:Delta_k-p=infty,d>2 }
Let the generating weights $w$ and $v$ be given by \eqref{w(x),p=infty} and \eqref{v(x),p=infty}. 	Let  $1\le q \le \infty$ and $r_{\lambda,\infty,q} >0$.  Then we have that
	for every $ f \in W^r_{\infty,w}(\RRd)$,
	\begin{equation}\label{Series2-p=infty,d>2}
		f
		= 	\sum_{\bk \in \NNd_0}\Delta_\bk (f)
	\end{equation}
	with absolute convergence in the space $L_{q,w}(\RR)$  of the series. 
	Moreover,
	\begin{equation}\label{Delta_kf-p=infty,d>2}
		\big\|\Delta_\bk f\big\|_{L_{q,w}(\RRd)}  
		\leq C 2^{-  r_{\lambda,\infty,q} |\bk|_1} |\bk|_1^d\,\|f\|_{W^r_{\infty,w}(\RRd)}, 
		\ \  \bk \in \NNd_0.
	\end{equation}
\end{lemma}

In the next step, similarly to Lemma \ref{lemma:I_xi-error}, this lemma implies

\begin{lemma} \label{lemma:I_xi-error-p=infty}
	Under the assumptions of Lemma 	\ref{lemma:Delta_k-p=infty,d>2 }
we have that	
	\begin{equation}\label{upperbound-p=infty}
		\big\|f - P_m (f)\big\|_{L_{q,w}(\RRd)}  
		\ll 
		\|f\|_{W^r_{\infty,w}(\RRd)}
			2^{- r_{\lambda,\infty,q}m} m^{2d - 1},
		\ \ m > 1, \  f \in W^r_{\infty,w}(\RRd). 
	\end{equation}
\end{lemma}

Analogously to Theorem \ref{thm:rho_n(W)}, from Lemmata \ref{lemma:Delta_k-p=infty,d>2 } and \ref{lemma:I_xi-error-p=infty} we deduce the following results.
\begin{theorem} \label{theorem:Int_n(W)-p=infty}
	Under the assumptions of Lemma 	\ref{lemma:Delta_k-p=infty,d>2 }, 	denote 
	$$
	\varrho_n:= \varrho_n(\bW^r_{\infty,w}(\RRd), L_{q,w}(\RRd)).
	$$  
	For every $n \in \NN$, let $m_n$ be the largest number such that $|G(m_n)| \le n$.	  Then  $P_{m_n}$ defines a sampling operator belonging to $\Ss_n$, and we have that
	\begin{equation}\label{rho_n(W)-p=infty}
		n^{-r_{\lambda,\infty,q}}
		\ll
		\varrho_n
		\le
	\sup_{f\in \bW^r_{\infty,w}(\RRd)}	
\big\|f - P_{m_n} (f)\big\|_{L_{q,w}(\RRd)}   		
		\ll 
			n^{-r_{\lambda,\infty,q}} (\log n)^{r_{\lambda,\infty,q}(d-1) + 2d -1}.
	\end{equation}	
\end{theorem}

\begin{figure}
	\centering{
		\begin{tabular}{cc}
			\includegraphics[height=7.0cm]{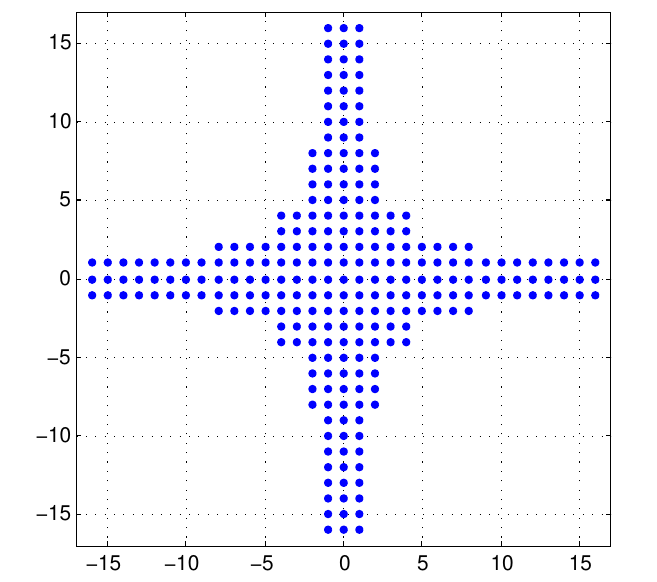}	 &  
			\includegraphics[height=7.0cm]{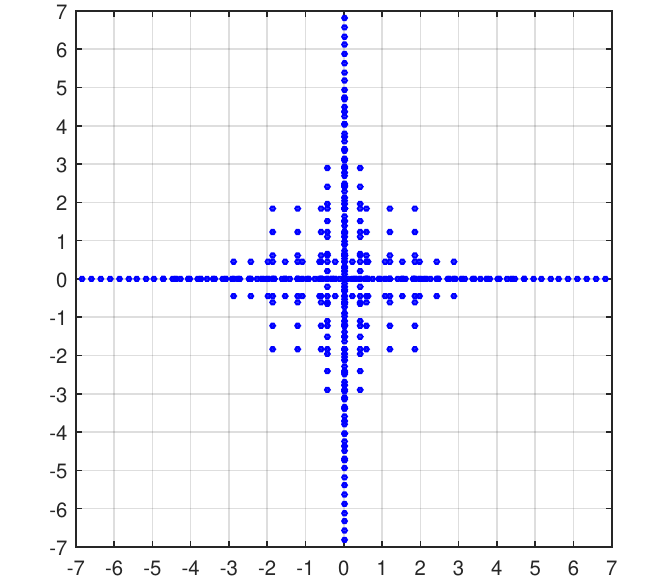}
			\\
			{A classical HC in frequency domain} & {A Hermite HC in function domain}
		\end{tabular}
	}
	\caption{Different hyperbolic crosses (HC)($d=2$) borrowed from \cite{DD2023}}
	\label{Fig-1}
\end{figure}
\begin{figure}
	\centering{
		\begin{tabular}{cc}
			\includegraphics[height=7.0cm]{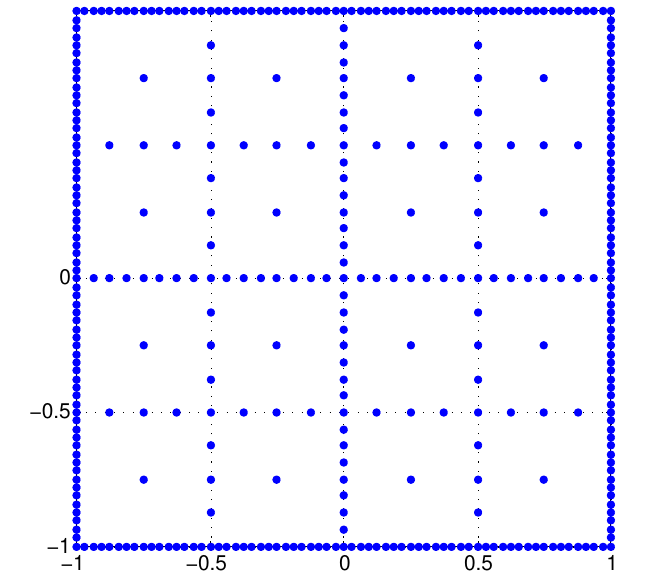}	 &  
		\end{tabular}
	}
	\caption{A classical Smolyak grid ($d=2$) borrowed from \cite{DD2023}}
	\label{Fig-2}
\end{figure}

		The grids of sample points $H(m)$ in the  algorithm $P_m$, which are formed from the non-equidistant zeros of the orthonornal polynomials $p_m(w)$, as noticed above,
		is a  step hyperbolic cross   on the function domain $\RRd$.  The grids  $H(m)$  completely differ from classical Smolyak grids  on the function domain $[-1,1]^d$ which are used in sampling recovery  for functions having a mixed smoothness (see, e.g., \cite[Section~5.3]{DTU18B} for detail). In Figure~\ref{Fig-1},  a step hyperbolic cross grid in the two-dimensional function domain  $\RR^2$ in the right picture is designed for the weight $w(\bx)= \exp(- x_1^2 - x_2^2)$, and  a classical step hyperbolic cross in the two-dimensional frequency domain is in the left picture. The classical Smolyak  grid on the domain $[-1,1]^2$ is in the left picture  of Figure \ref{Fig-2}.
	 The grids  $H(m)$ are very sparsely distributed inside the $d$-cube 
		$$K(m): = \brab{\bx \in \RRd: \, |x_i| \le C2^{m/\lambda}, \ i =1,...,d},$$
		for some constant $C > 0$.
		Its diameter  which is the length of its symmetry axes is  $2C2^{m/\lambda}$, i.e.,  the size of $K(m)$. 

\appendix
\section{Appendix} \label{appendix}

For $\bp = (p_1,...,p_d) \in [1,\infty)^d$, we defined the mixed integral  norm 
$\|\cdot\|_{L_{\bp,w}(\RRd)}$ for functions on $\RRd$ as follows
\begin{equation*}
	\|f\|_{L_{\bp,w}(\RRd)}  
	\ := \
	\left( \int_{\RR} \left(\cdots \int_{\RR}\left(\int_{\RR}|f(\bx) w(\bx)|^{p_1} \rd x_1\right)^{p_2/p_1} 
	\rd x_2 \cdots \right)^{p_d/p_{d-1}} \rd x_d \right)^{1/p_d},
\end{equation*}
and put $\bone/{\bp} := (1/p_1,...,1/p_d)$. For $\bm \in \NNd_0$, denote by $\Pp_\bm$ the space of polynomials on $\RRd$ of degree at most $m_j$ in the variable $x_j$, $j =1,...,d$.
If ${\bp},{\bq} \in [1,\infty)^d$, from Lemma \ref{lemma:B-NInequality} one can deduce the
Nikol'skii's inequalities
\begin{equation} \label{ineq:Nikol'skii p<q}
	\|\varphi\|_{L_{\bq,w}(\RRd)} 
	\  \le \
	C 2^{(1 - 1/\lambda)|(\bone/{\bp} - \bone/{\bq})\bk|_1} \|f\|_{L_{\bp,w}(\RRd)} \ \ \ \forall \varphi \in \Pp_{2^\bk} \ \ \ ({\bp} < {\bq}),
\end{equation}
and
\begin{equation} \label{ineq:Nikol'skii p>q}
	\|\varphi\|_{L_{w,\bq}(\RRd)} 
	\  \le \
	C 2^{(1/\lambda)|(\bone/{\bq} - \bone/{\bp})\bk|_1} \|f\|_{L_{\bp,w}(\RRd)} \ \ \ \forall \varphi \in \Pp_{2^\bk}
 \ \ \ ({\bp}  > {\bq}),
 \end{equation}
with constants $C$ depending on $\lambda, {\bp},{\bq},d$ only, where we denote
$\bx \by:= (x_1y_1,...,x_d y_d)$.

\begin{lemma}  \label{Lemma:IneqInt[varphi_k*varphi_s]}
	Let $1\le  p , q < \infty$, $p \not= q$ and $\tau := |1/2 - p/(p+q)|$. 
	Then there holds the inequality 
	\begin{equation*}  
		\int_{\RRd} |\varphi_\bk(\bx) \varphi_\bs(\bx) |^{q/2} w(\bx) \rd \bx
		\ \le \ 
		C S_\bk S_\bs 2^{- \tau |\bk - \bs|_1} \ \ \ 
		\forall \varphi_\bk  \in \Pp_{2^\bk},  \ \varphi_\bs  \in \Pp_{2^\bs},  \ \bk, \bs \in \NNd_0,
	\end{equation*}
	with some constant $C$ depending at most on $\lambda, p,q,d$, where
	\begin{equation*}  
		S_\bk := \ \left(2^{\delta_{\lambda,p,q}|\bk|_1} \|\varphi_\bk \|_{L_{p,w}(\RRd)}\right)^{q/2}.
	\end{equation*} 
\end{lemma}

\begin{proof} We prove the lemma for the case $p > q$. 
	It can be proven in a similar way with some modification for the $p < q$.
	Let $\nu := (p + q)/p$. Then $\tau = 1/\nu - 1/2 $ and $1< \nu < 2$.
	Put $\nu':= \nu/(\nu - 1)$, we have $1/\nu + 1/\nu' = 1$ and $1 < \nu' < 2 $.
	Let ${\bu}, {\bv} \in (1,\infty)^d$ be defined by ${\bu} := q{\bv}/2$ and 
	$v_i = \nu$ if $k_i <  s_i$, and $v_i = \nu'$ if $k_i \ge s_i$ for $i= 1,...,d$.  
	Let ${\bu}'$ and ${\bv}'$ be given by $\bone/{\bu}+ \bone/{\bu}' = \bone$
	and $\bone/{\bv}+ \bone/{\bv}' = \bone$, respectively.
	Notice that ${\bv} \in (1,\infty)^d$.
	Applying the H\"older  inequality for the mixed norm $\|\cdot\|_{L_{w,\bv}(\RRd)}$, we obtain 
	\begin{equation}  \label{ineq:Int[varphi_k*varphi_s]}
		\begin{split}			
		\int_{\RRd} |\varphi_\bk(\bx) \varphi_\bs(\bx)|^{q/2} w(\bx)\rd \bx
		\ &\le \ 
		\||\varphi_\bk|^{q/2}\|_{L_{\bv,w}(\RRd)} \||\varphi_\bs|^{q/2}\|_{L_{\bv',w}(\RRd)} \\
		\ &= \ 
	\|\varphi_\bk\|_{L_{\bu,w}(\RRd)}^{q/2} \|\varphi_\bs\|_{L_{\bu',w}(\RRd)}^{q/2}.	
\end{split}
	\end{equation}
	Since ${\bu} < p{\bf 1}$ and ${\bu}' < p{\bf 1}$, by inequality \eqref{ineq:Nikol'skii p>q} we have
	\begin{equation} \label{ineq:Norms[varphi_k,varphi_s]}
		\begin{split}
			\|\varphi_\bk\|_{L_{\bu,w}(\RRd)}
		\  &\le \
		2^{(1/\lambda)|(\bone/{\bu} - \bone/p)\bk|_1} \|\varphi_\bk\|_{L_{p,w}(\RRd)}, \\
		\|\varphi_\bs\|_{L_{\bu',w}(\RRd)}
		\  &\le \
		2^{(1/\lambda)|(\bone/{\bu'} - \bone/p)\bs|_1} \|\varphi_\bs\|_{L_{p,w}(\RRd)}.  
\end{split}
	\end{equation}
	From \eqref{ineq:Int[varphi_k*varphi_s]} and 
	\eqref{ineq:Norms[varphi_k,varphi_s]} we prove the lemma.
	\hfill
\end{proof}

\begin{lemma}  \label{lemma:IneqL_qNorm<L_pNorm}
	Let $1 \le  p, q < \infty$, $p \not= q$ and  
	$f \in L_{q,w}(\RRd)$ be represented by  the  series 
	\begin{equation*} 
		f \ = \sum_{\bk \in \NNd_0} \ \varphi_\bk, \ \varphi_\bk \in \Pp_{2^\bk}
	\end{equation*}
	converging in $L_{q,w}(\RRd)$.
	Then there holds the inequality 
	\begin{equation} \label{ineq:IneqL_qNorm<L_pNorm}
		\|f\|_{L_{q,w}(\RRd)}
		\ \le C \left( \sum_{k \in {\ZZ}^d_+} \ \|2^{\delta_{\lambda,p,q}|\bk|_1}  \varphi_\bk \|_{L_{p,w}(\RRd)}^q \right)^{1/q}, 
	\end{equation}
	with some constant $C$ depending at most on $\lambda,p,q,d$, whenever the right side is finite. 
\end{lemma}

\begin{proof} 
	It is sufficient to show the inequality \eqref{ineq:IneqL_qNorm<L_pNorm} for $f$ of the form
	\begin{equation*} 
		f \ = \sum_{\bk \le \bm} \ \varphi_\bk, \ \varphi_\bk \in \Pp_{2^\bk},
	\end{equation*} 
	for any $\bm \in \NNd_0$. We will prove this  for the case $1 \le  q < p < \infty$.   The case $1 \le  p < q < \infty$ can be proven analogously.
	
	Put $n:= [q] + 1$. Then $0 < q/n \le 1$. By the Jensen  inequality we have
	\begin{equation*}
		\begin{aligned}  
			\left|\sum_{\bk \le \bm} \ \varphi_\bk(\bx)\right|^q
			\ & = \
			\left(\left|\sum_{\bk \le \bm} \ \varphi_\bk(\bx)\right|^{q/n}\right)^n \\
			\  \le \
			\left(\sum_{\bk \le \bm} \ |\varphi_\bk(\bx)|^{q/n}\right)^n 
			\ & = \
			\sum_{\bk^1 \le m} \cdots  \sum_{\bk^n \le m}\ \prod_{j=1}^n |\varphi_{\bk^j}(\bx)|^{q/n}. \\  
		\end{aligned}
	\end{equation*}
	and therefore,
	\begin{equation} \label{ineq:[|f|_q^q](1)}
		\|f\|_{L_{q,w}(\RRd)}^q
		\  \le \
		\sum_{\bk^1 \le \bm} \cdots  \sum_{\bk^n \le \bm}\ \int_{\RRd} \ 
		\prod_{j=1}^n |\varphi_{\bk^j}(\bx)|^{q/n} w(\bx) \rd \bx.   
	\end{equation}
From the identity
	\begin{equation} \label{ineq:prod_{j=1}^n a_j}
		\prod_{j=1}^n a_j 
		\ = \
		\left( \prod_{i \ne j} a_i a_j \right)^{1/2(n-1)} 
	\end{equation}
	for non-negative numbers $a_1,..., a_n$, we obtain
	\begin{equation*} 
		\Ii := \ \int_{\RRd} \ \prod_{j=1}^n |\varphi_{\bk^j}(\bx)|^{q/n} w(\bx) \rd \bx
		\ = \ 
		\int_{\RRd} \prod_{i \ne j}|\varphi_{\bk^i}(\bx)\varphi_{\bk^j}(\bx)|^{q/2n(n-1)} w(\bx) \rd \bx.   
	\end{equation*}
	Hence, applying the H\"older inequality to $n(n-1)$ functions 
	in the right side of the last equality, Lemma \ref{Lemma:IneqInt[varphi_k*varphi_s]} and 
	\eqref{ineq:prod_{j=1}^n a_j} give
	\begin{equation*}
		\begin{aligned}  
			\Ii
			\ & \le \ 
			\prod_{i \ne j} \left( \int_{\RRd} |\varphi_{\bk^i}(\bx)\varphi_{\bk^j}(\bx)|^{q/2} w(\bx) \rd \bx \right)^{1/n(n-1)} 
			\  \le \ 
			\prod_{i \ne j} \left( S_{\bk^i} S_{\bk^j} 2^{- \tau |\bk^i - \bk^j|_1} \right)^{1/n(n-1)} \\   
			\ & = \ 
			\prod_{i \ne j} \left( S_{\bk^i} S_{\bk^j} \right)^{1/n(n-1)}
			\left\{\left( \prod_{i \ne j}\prod_{i' = 1}^n 2^{- \tau |\bk^i - \bk^{i'}|_1}
			\prod_{j' = 1}^n 2^{- \tau |\bk^j - \bk^{j'}|_1}\right)^{1/2(n-1)}\right\}^{1/n(n-1)} \\  
			\ & = \ 
			\left\{\prod_{i \ne j} \left( S_{\bk^i} S_{\bk^j} \right)^{2/n}
			\left( \prod_{i' = 1}^n 2^{- \tau |\bk^i - \bk^{i'}|_1}
			\prod_{j' = 1}^n 2^{- \tau |\bk^j - \bk^{j'}|_1}\right)^{1/n(n-1)}\right\}^{1/2(n-1)} \\
			\ & = \ 
			\prod_{j = 1}^n S_{\bk^j}^{2/n}
			\left(\prod_{i = 1}^n 2^{- \tau |\bk^j - \bk^i |_1}\right)^{1/n(n-1)} 
			\  = \ 
			\left(\prod_{j = 1}^n S_{\bk^j}^2
			\prod_{i = 1}^n 2^{- \eta \tau |\bk^j - \bk^i |_1}\right)^{1/n},  
		\end{aligned}  
	\end{equation*}
	where $\eta := \tau/(n-1) > 0$. (Here and below we use the notation in Lemma \ref{Lemma:IneqInt[varphi_k*varphi_s]}.)
	Therefore, from \eqref{ineq:[|f|_q^q](1)} and the H\"older inequality we obtain
	\begin{equation} \label{ineq:[|f|_q^q](2)} 
		\begin{aligned}  
			\|f\|_{L_{q,w}(\RRd)}^q
			\ & \le \
			\sum_{\bk^1 \le \bm} \cdots  \sum_{\bk^n \le \bm}\ 
			\left(\prod_{j = 1}^n S_{\bk^j}^2
			\prod_{i = 1}^n 2^{- \eta \tau |\bk^j - \bk^i |_1}\right)^{1/n} \\
			\ & \le \
			\prod_{j = 1}^n 
			\left(\sum_{\bk^1 \le \bm} \cdots  \sum_{\bk^n \le \bm}\ 
			S_{\bk^j}^2
			\prod_{j = 1}^n 2^{- \eta \tau |\bk^j - \bk^i |_1}\right)^{1/n} 
			\ =: \ \prod_{j = 1}^n B_j.   
		\end{aligned}  
	\end{equation}
	We have
	\begin{equation*}
		\begin{aligned}  
			B_j 
			\ & = \
			\sum_{\bk^j \le \bm} \ S_{\bk^j}^2 
			\sum_{\bk^1 \le \bm} \cdots  \sum_{\bk^{j-1} \le \bm} \ \sum_{\bk^{j+1} \le \bm} \cdots  \sum_{\bk^n \le \bm} 
			\prod_{i = 1}^n 2^{- \eta \tau |\bk^j - \bk^i |_1} \\
			\ & = \
			\sum_{\bk^j \le \bm} \ S_{\bk^j}^2 \left( \sum_{\bs \le \bm}  
			2^{- \eta \tau |\bk^j - \bs|_1} \right)^{n-1}
			\ \le \
			C \sum_{\bk^j \le \bm} \ S_{\bk^j}^2.
		\end{aligned}  
	\end{equation*}
	Using the last bound for $B_j$, continuing the estimates \eqref{ineq:[|f|_q^q](2)} we  finish the proof of the lemma:
	\begin{equation*}
		\begin{aligned}  
			\|f\|_{L_{q,w}(\RRd)}^q 
			\ & \le \
			\prod_{j = 1}^n B_j^{1/n} 
			\ \le \
			C \sum_{\bk \le \bm} \ S_\bk^2 \\
			\ & = \
			C \sum_{\bk \le \bm} \ \|2^{\delta_{\lambda,p,q}|\bk|_1} \varphi_\bk \|_{L_{p,w}(\RRd)}^q .
		\end{aligned}  
	\end{equation*}
\hfill
\end{proof}

In the proofs of Lemma \ref{lemma:IneqL_qNorm<L_pNorm}, we have used a modification of a technique employed in \cite{Tem85} and \cite{Dung11a}  in the proofs  of a trigonometric version ($1 \le p < q < \infty$)  and of  a B-spline version ($0 < p < q < \infty$) of the results.

\medskip
\noindent
{\bf Acknowledgments:} 
This work is funded by the Vietnam National Foundation for Science and Technology Development (NAFOSTED) in the frame of the NAFOSTED-SNSF Joint Research Project under  
Grant IZVSZ2$_{ - }$229568.  
 A part of this work was done when  the author was working at the Vietnam Institute for Advanced Study in Mathematics (VIASM). He would like to thank  the VIASM  for providing a fruitful research environment and working condition. The author  gratefully acknowledges the referees for valuable comments and remarks that significantly improved the presentation of this paper.
 
\bibliographystyle{abbrv}
\bibliography{WeightedSampling}
\end{document}